\documentclass[11pt]{amsart}

\usepackage[bookmarksnumbered=true, bookmarksopen=true]{hyperref}
\usepackage{amsmath,amssymb,mathrsfs,amsthm}
\usepackage{enumerate,amsfonts,stmaryrd,graphicx,psfrag}
\usepackage[all]{xy}
\usepackage[margin=1in]{geometry}

\newtheorem{thm}{Theorem}[section]
\newtheorem{cor}[thm]{Corollary}
\newtheorem{cor*}[thm]{Corollary*}
\newtheorem{lem}[thm]{Lemma}
\newtheorem{prop}[thm]{Proposition}
\theoremstyle{definition}
\newtheorem{defn}[thm]{Definition}
\newtheorem{conj}[thm]{Conjecture}
\theoremstyle{remark}
\newtheorem{rem}[thm]{Remark}

\numberwithin{equation}{section}

\newcommand{\C}{\mathbb{C}}      
\newcommand{\Z}{\mathbb{Z}}      

\newcommand{\R}{\mathbb{R}}
\newcommand{\Q}{\mathbb{Q}}
\newcommand{\cO}{\mathcal{O}}
\newcommand{\Pj}{\mathbb{P}}
\newcommand{\cE}{\mathcal{E}}
\newcommand{\cF}{\mathcal{F}}

\newcommand{\fA}{\mathfrak{A}}
\newcommand{\fC}{\mathfrak{C}}
\newcommand{\fX}{\mathfrak{X}}
\newcommand{\fP}{\mathfrak{P}}
\newcommand{\fM}{\mathfrak{M}}
\newcommand{\fg}{\mathfrak{g}}
\newcommand{\fh}{\mathfrak{h}}
\newcommand{\fsl}{\mathfrak{sl}}
\newcommand{\fgl}{\mathfrak{gl}}
\newcommand{\fp}{\mathfrak{p}}
\newcommand{\cI}{\mathcal{I}}

\newcommand{\fz}{\mathfrak{z}}
\newcommand{\cX}{\mathcal{X}}
\newcommand{\cY}{\mathcal{Y}}
\newcommand{\cD}{\mathcal{D}}
\newcommand{\cV}{\mathcal{V}}
\newcommand{\cL}{\mathcal{L}}
\newcommand{\cM}{\mathcal{M}}
\newcommand{\cH}{\mathcal{H}}
\newcommand{\cZ}{\mathcal{Z}}
\newcommand{\cA}{\mathcal{A}}
\newcommand{\cB}{\mathcal{B}}

\newcommand{\bI}{\mathbb{I}}
\newcommand{\bJ}{\mathbb{J}}
\newcommand{\bv}{\mathbf{v}}
\newcommand{\bw}{\mathbf{w}}
\newcommand{\bm}{\mathbf{m}}
\newcommand{\bzero}{\mathbf{0}}

\newcommand{\tr}{\operatorname{tr}}
\newcommand{\Id}{\operatorname{Id}}

\newcommand{\Aut}{\operatorname{Aut}}
\newcommand{\Hilb}{\operatorname{Hilb}}

\newcommand{\Ext}{\operatorname{Ext}}
\newcommand{\Hom}{\operatorname{Hom}}

\newcommand{\rar}{\rightarrow}
\newcommand{\xrar}{\xrightarrow}

\newcommand{\vir}{\mathrm{vir}}
\newcommand{\llb}{\llbracket}
\newcommand{\rrb}{\rrbracket}
\newcommand{\ch}{\widetilde{ch}}

\newcommand{\orb}{\mathrm{orb}}
\newcommand{\age}{\mathrm{age}}
\newcommand{\reg}{\mathrm{reg}}
\newcommand{\red}{\mathrm{red}}

\newcommand{\out}{\mathrm{out}}
\newcommand{\At}{\mathrm{At}}
\newcommand{\GW}{\mathrm{GW}}
\newcommand{\DT}{\mathrm{DT}}
\newcommand{\Stab}{\operatorname{Stab}}
\newcommand{\pt}{\operatorname{pt}}
\newcommand{\ev}{\operatorname{ev}}
\newcommand{\Ob}{\operatorname{Ob}}

\allowdisplaybreaks

\begin{document}

\nocite{*}

\title{Donaldson--Thomas theory of $[\C^2/\Z_{n+1}]\times \Pj^1$}

\author{Zijun Zhou}
\address{Zijun Zhou, Department of Mathematics, Columbia University, New York, NY 10027, USA}
\email{zzhou@math.columbia.edu}

\maketitle

\begin{abstract}
  We study the relative orbifold Donaldson--Thomas theory of $[\C^2/\Z_{n+1}]\times \Pj^1$. A correspondence is established between the DT theory relative to disjoint union of vertical fibers to quantum multiplication by divisors for the Hilbert scheme of points on $[\C^2/\Z_{n+1}]$. This determines a correspondence between the whole theories if a further nondegeneracy condition is assumed. The result can also be viewed as a crepant resolution correspondence to the DT theory of $\cA_n\times \Pj^1$.
\end{abstract}

\tableofcontents

\section{Introduction}

\subsection{GW/DT correspondence and triangle of equivalences}

In \cite{MNOP1, MNOP2}, D. Maulik, N. Nekrasov, A. Okounkov and R. Pandharipande established remarkable conjectures, which relates the Gromov--Witten theory and Donaldson--Thomas theory of a Calabi--Yau 3-fold $X$, known as GW/DT correspondence. The correspondence states that certain generating functions of GW invariants and DT invariants of $X$ can be equated to each other after a change of variables.

An important example of this correspondence lives in the non-CY and non-compact world. Let $C$ be an algebraic curve of genus $g$, and $L_1, L_2$ be line bundles on $C$. The total space of $L_1\oplus L_2$ is considered as a typical \emph{local curve}. Under degenerations of $C$, the GW/DT theory of local curves reduces to the relative GW/DT theory of $\C^2\times \Pj^1$ with respect to some fibers, together with some other special cases. Bryan--Pandharipande \cite{BP} and Okounkov--Pandharipande \cite{OP-DT} studied the equivariant GW and DT theory of local curves relative to fibers, and found that they satisfied the GW/DT correspondence. The way they prove this correspondence is to work with a third theory, the small quantum cohomology of the Hilbert scheme of points on $\C^2$, computed by Okounkov--Pandharipande \cite{OP-QH}.

Let the cyclic group
$$\Z_{n+1}:= \Z/(n+1)\Z = \{\zeta\in \C | \zeta^{n+1}=1 \}$$
act on $\C^2$ in the anti-diagonal manner
$$\zeta\cdot (x,y):= (\zeta x, \zeta^{-1} y).$$
The surface $\cA_n$ is defined to be the minimal resolution of the singular surface $\C^2/\Z_{n+1}$.

Maulik \cite{Mau} and Maulik--Oblomkov \cite{MOb09-QH, MOb09-DT} generalized the picture described above to $\cA_n$, and forming the following triangle of equivalences.
$$\xymatrix{
& \text{QH}(\Hilb(\cA_n)) \ar@{-}[dl] \ar@{-}[dr] & \\
\text{GW}(\cA_n\times \Pj^1) \ar@{-}[rr] & & \text{DT}(\cA_n\times \Pj^1).
}$$
This triangle has been used in Maulik--Oblomkov--Okounkov--Pandharipande \cite{MOOP} to prove the GW/DT correspondence for general toric 3-folds. The key techniques adopted in computations of these theories are equivariant localization and the degeneration formula, which are both available on the GW and DT side.

\subsection{Crepant resolutions, Nakajima quiver varieties and flops}

The resolution of singularity $f: \cA_n \rar \C^2/\Z_{n+1}$, in particular, is a crepant resolution, meaning that $f$ preserves the canonical divisor. The 3-(orbi)folds $\cA_n\times \Pj^1$ and $[\C^2/\Z_{n+1}]\times \Pj^1$ are therefore in the position of the crepant resolution conjecture, proposed by Ruan \cite{Ruan, BG}, which says that there should be a correspondence between the GW (resp. DT) theory of $\cA_n\times \Pj^1$, and the \emph{orbifold} GW (resp. DT) theory of $[\C^2/\Z_{n+1}]\times \Pj^1$.

We are naturally motivated to consider the following triangular prism of equivalences
$$\xymatrix{
& \text{QH}(\Hilb(\cA_n)) \ar@{-}[dl] \ar@{-}[dr] \ar@{-}'[d][dd] & \\
\text{GW}(\cA_n\times \Pj^1) \ar@{-}[rr] \ar@{--}[dd] & & \text{DT}(\cA_n\times \Pj^1) \ar@{-}[dd] \\
& \text{QH}(\Hilb([\C^2/\Z_{n+1}])) \ar@{--}[dl] \ar@{-}[dr] & \\
\text{GW}([\C^2/\Z_{n+1}]\times \Pj^1) \ar@{--}[rr] & & \text{DT}([\C^2/\Z_{n+1}]\times \Pj^1),}$$
where the lower triangle is about the quantum cohomology of $\Hilb([\C^2/\Z_{n+1}])$, and the relative orbifold GW and DT theory of $[\C^2/\Z_{n+1}]\times \Pj^1$. In particular, an orbifold GW/DT correspondence is expected. The vertical lines for GW and DT can be viewed as crepant resolution results.

For the vertical line in the middle, we need the recent work by Maulik--Okounkov \cite{MOk12}, in which a systematic way is introduced to compute the quantum multiplication by divisors for Nakajima quiver varieties.

Nakajima quiver varieties defined by the same quiver but different (generic) stability conditions are related via flops, in the same manner as varying stability conditions in the GIT theory. In the special case of Hilbert schemes of points on surfaces, A. Kuznetsov \cite{Ku} and K. Nagao \cite{Nag} described the construction of $\Hilb^m(\cA_n)$ and $\Hilb^m([\C^2/\Z_{n+1}])$ as quiver varieties under flops, where the latter can be understood as $\Z_{n+1}$-equivariant Hilbert scheme of points on $\C^2$.

Curve classes on those Hilbert schemes can be identified with roots of a certain Kac--Moody algebra, and the stability condition determines a certain chamber in the root space, which corresponds to the effective curves. The relationship between the quantum cohomology of $\Hilb(\cA_n)$ and $\Hilb^m([\C^2/\Z_{n+1}])$ can therefore be interpreted by wall-crossing, under a simple change of stability conditions, or chambers.

\subsection{Statements of results}

In this paper we mainly concentrate on the computation of the relative DT theory of $\cX:=[\C^2/\Z_{n+1}]\times \Pj^1$. In other words, we complete the line between QH and DT in the lower triangle, and also the vertical DT line. We use $A^*$ instead of $H^{2*}$ to denote cohomologies with $\Q$ coefficients, just to simplify the degrees.

Let $\cY := B\Z_{n+1} \times \Pj^1$ and $p\in \Pj^1$ be a point. Let $\rho_j$ be the 1-dimensional representation of $\Z_{n+1}$ with eigenvalue $\zeta^j$, and $\rho_\reg:= \sum_{j=0}^n \rho_j$ be the regular representation. Consider a class $P= m [\rho_\reg\otimes \cO_\cY] + \sum_{j=0}^n \varepsilon_j [\rho_j\otimes \cO_p]$ be in the topological K-theory of $\cX$, where $m, \varepsilon_j \in \Z$. Let $N_i= [\C^2/\Z_{n+1}]\times \{p_i\}$ be vertical fibers. The moduli space of relative DT theory is $\Hilb^P(\cX, \coprod N_i)$, which is a Deligne--Mumford stack, parameterizing closed substacks $\cZ\subset \cX$ satisfying certain transversality conditions with respect to the relative divisors $N_i$ and also some stability conditions.

Relative descendant DT invariants
$$\langle \sigma_{k_1}(\gamma_1)\cdots \sigma_{k_l}(\gamma_l) | \xi_1,\cdots,\xi_r \rangle_m = \sum_{\varepsilon} q^{\varepsilon} \langle \sigma_{k_1}(\gamma_1)\cdots \sigma_{k_l}(\gamma_l) | \xi_1,\cdots,\xi_r \rangle_{m,\varepsilon}$$
are defined by intersection theory on the relative Hilbert stacks. Details will be given in Section 3.1 and 3.2. Here $\gamma_i\in A^*_\orb(\cX)$ and $\xi_i\in A^*(\Hilb^m([\C^2/\Z_{n+1}]))$ are certain cohomology classes. We also define the \emph{reduced} DT invariants
$$\langle \sigma_{k_1}(\gamma_1)\cdots \sigma_{k_l}(\gamma_l) | \xi_1,\cdots,\xi_r \rangle'_m,$$
after a quotient of the degree 0 contribution. In case without descendent insertions, we will abbreviate the notation as
$$\langle \xi_1, \cdots, \xi_r \rangle'_m : = \langle \ | \xi_1,\cdots,\xi_r \rangle'_m.$$

Our main theorem is the following. Let $\langle \xi_1, \cdots, \xi_r  \rangle^\GW_{\Hilb^m}$ denote the Gromov--Witten invariants of $\Hilb^m([\C^2/\Z_{n+1}])$ with primary insertions $\xi_i$'s.

\begin{thm}
  Let $A, B, \gamma\in A^*_T(\Hilb^m([\C^2/\Z_{n+1}]))$, where $\gamma$ is the fundamental class or a divisor. Then
  $$\langle A, B, \gamma \rangle'_m = \langle A, B, \gamma \rangle^\GW_{\Hilb^m}.$$
\end{thm}

As a result of this DT/Hilb correspondence, we obtain explicit formulas for  relative DT 3-point functions, using the quantum multiplication formula of Maulik--Okounkov \cite{MOk12}. This allows us to compare our results with those of Maulik--Oblomkov \cite{MOb09-DT} on $\cA_n\times \Pj^1$ and obtain a relative crepant resolution correspondence.

Recall that there is an explicit isomorphism between cohomology rings
$$A^*_\orb([\C^2/\Z_{n+1}]) \cong A^*(\cA_n),$$
$$e_0 \mapsto 1, \qquad e_i \mapsto \frac{\zeta^{i/2} - \zeta^{-i/2}}{n+1} \sum_{j=1}^n \zeta^{ij} \omega_j, \qquad 1\leq i\leq n,$$
where $\omega_1, \cdots, \omega_n \in H^2(\cA_n, \Q)$ is the dual basis to the exceptional curves in $\cA_n$. Under this isomorphism, we can explicitly identify the Fock spaces $A^*(\Hilb^m([\C^2/\Z_{n+1}])) \cong A^*(\Hilb^m(\cA_n))$.

For a curve $Z\subset \cA_n\times \Pj^1$, its topological data is specified by the pair $(\chi, (\beta, m))$, where $\chi = \chi(\cO_Z)\in \Z$ and $\beta\in H_2(\cA_n, \Z)$ such that $m[\Pj^1] + \beta = [Z]\in H_2(\cA_n\times \Pj^1, \Z)$. The generating function for the relative DT theory of $\cA_n\times \Pj^1$ is defined as
$$\langle \xi_1, \cdots, \xi_r \rangle_{\cA_n\times \Pj^1, m} := \sum_{\chi, \beta} Q^\chi Q_1^{(\beta, \omega_1)} \cdots Q_n^{(\beta, \omega_n)} \langle \xi_1, \cdots, \xi_r \rangle_{\cA_n\times \Pj^1, \chi, (\beta, m)},$$
and the \emph{reduced} partition function $Z'_\DT$ is defined by quotient out the degree $0$ contribution. We have the following reselt.

\begin{thm}
  Let $A, B, \gamma\in A_{T^\pm}^*(\Hilb^m([\C^2/\Z_{n+1}]))$, where $\gamma$ is the fundamental class or a divisor. Then
  $$\langle A, B, \gamma \rangle'_{\cX, m} = Q^{-m} \langle A, B, \gamma \rangle'_{\cA_n\times \Pj^1, m},$$
  under the identification
  $$Q\mapsto q_0 q_1 \cdots q_n, \qquad Q_i \mapsto q_i, \qquad i\geq 1,$$
  and analytic continuation.
\end{thm}

We also compute the equivariant degree 0 invariants. Let $s_1, s_2, s_3$ be the tangent weights of the torus action in $[\C^2/\Z_{n+1}]\times \C$.  The $\Z_{n+1}$-colored multi-regular equivariant vertex is defined as the generating function over all $\Z_{n+1}$-colored 3d partitions $\pi$:
$$Z_{\Z_{n+1}} (s,q) := \sum_\pi w(\pi) q_0^{|\pi|_0}\cdots q_n^{|\pi|_n} \quad\in\quad \Q(s_1,s_2,s_3) \llb q_0,\cdots,q_n \rrb,$$
where $|\pi|_i$ is the number of boxes in $\pi$ of color $i$, and $w(\pi)$ is a $\Z_{n+1}$-version of the equivariant vertex measure in the sense of \cite{MNOP1, MNOP2}.

\begin{thm}
  The $(\C^*)^3$-equivariant DT vertex for $[\C^2/\Z_{n+1}]\times \C$ is
  $$Z_{\Z_{n+1}} (s,q) = M(1,-Q)^{-\frac{(n+1)(s_1+s_2)}{s_3} - \frac{(s_1+s_2)(s_1+s_2+s_3)}{(n+1)s_1s_2}} \prod_{1\leq a\leq b\leq n} \left( M(q_{[a,b]}, Q) M(q_{[a,b]}^{-1}, Q) \right)^{-\frac{s_1+s_2}{s_3}},$$
  where $q_{[a,b]} = q_a\cdots q_b$, $Q=q_0 q_1\cdots q_n$, $M(x,Q) = \prod_{k\geq 1} \left( 1-xQ^k \right)^{-k}$.
\end{thm}

\begin{thm}
  The $T$-equivariant $k$-point relative degree 0 DT invariant of $[\C^2/\Z_{n+1}]\times \Pj^1$ is
  $$\langle \ | \emptyset, \cdots, \emptyset \rangle_{\cX, m=0} = M(1,-Q)^{ (k-2)\cdot \frac{(s_1+s_2)^2}{(n+1)s_1s_2}},$$
  where $Q= q_0 q_1 \cdots q_n$.
\end{thm}

The paper is organized as follows. In Section 2 we describe the geometry of the orbifold surface $[\C^2/\Z_{n+1}]$ and the Hilbert scheme of points on it. We use a diffeomorphism between $\Hilb^m([\C^2/\Z_{n+1}])$ and $\Hilb^m(\cA_n)$ to introduce the Nakajima basis, and then the $\widehat{\fgl}(n+1)$-action, on the cohomology of $\Hilb^m([\C^2/\Z_{n+1}])$. In Section 3 we define the orbifold relative DT invariants, and in our special case, introduce a reduced obstruction theory, which is special because of the symplectic structure on $[\C^2/\Z_{n+1}]$. Section 4 involves some intermediate geometry we will need. We relate certain descendant invariants to those on the rubber moduli space. Section 5 computes the degree 0 invariants. Section 6 computes the cap and tube invariants, which are relatively simple. In section 7, we compute the 3-point DT invariants, in which one of the insertion is a divisor. A comparison with the quantum cohomology of the Hilbert scheme of points is crucial in this calculation.

\subsection{Acknowledgements}

The author would like to thank Davesh Maulik, Hiraku Nakajima, Andrei Okounkov, Amdrey Smirnov, Changjian Su, Richard Thomas and Jingyu Zhao for helpful discussions and suggestions. Moreover, the author would like to express his acknowledgements to Professor Chiu--Chu Melissa Liu, for useful conversations and suggestions. The project would not have been possible without her guidance and encouragement.

\section{Geometry of $[\C^2/\Z_{n+1}]$ and $\Hilb^m ([\C^2/\Z_{n+1}])$}

\subsection{Cyclic quotient of $\C^2$}

Let $[\C^2/\Z_{n+1}]$ be the group quotient of $\C^2$ by the cyclic group
$$\Z_{n+1}= \Z/(n+1)\Z= \{\zeta\in \C|\zeta^{n+1}=1\},$$
where the generator $\zeta:= e^{2\pi i /(n+1)}$ acts as $\zeta\cdot (x,y)= (\zeta x, \zeta^{-1}y)$. It is an orbifold surface, or more precisely, a 2-dimensional smooth Deligne--Mumford stack.

The coarse moduli space of the orbifold surface is
$$\xymatrix{
c: [\C^2/\Z_{n+1}] \ar[r] &  \C^2/\Z_{n+1},
}$$
where $\C^2/\Z_{n+1}$ is the affine GIT quotient. By definition, the homology and cohomology groups of $[\C^2/\Z_{n+1}]$ are identified with those of the coarse moduli space.

There is a torus action by $T:= (\C^*)^2$ on $\C^2$, defined by $(t_1, t_2)\cdot (x,y)= (t_1 x, t_2 y)$, which commutes with the $\Z_{n+1}$-action and induces a torus action on the quotient. The $T$-equivariant cohomology $A_T^*([\C^2/\Z_{n+1}])$ is a ring over $A_T^*(\pt)= \C[s_1, s_2]$, where the tangent weights of the two axes in $[\C^2/\Z_{n+1}]$ are $-s_1$ and $-s_2$. The origin represents an equivariant cohomology class
$$[0]= s_1 s_2 \in A_T^2 ([\C^2/\Z_{n+1}]).$$

We denote the 1-dimensional anti-diagonal torus by
$$T^\pm:= \{ (t,t^{-1}) \in T \},$$
and its corresponding equivariant parameter by $s$. There is a canonical reduction map from $T$ to $T^\pm$, identifying $A^*_{T^\pm}$ with the quotient of $A^*_T$ modulo $(s_1+s_2)$.

To reflect the stacky features of $[\C^2/\Z_{n+1}]$, we describe its orbifold cohomology. The inertia stack of an arbitrary DM stack $\cX$ is defined to be
$$I\cX:= \cX\times_{\cX\times \cX} \cX,$$
where the two projections $\cX\rar \cX\times \cX$ are the diagonal map. An $S$-point of $I\cX$ consists of the data $(f, s)$, where $f: S\rar \cX$ is a map to $\cX$, and $s\in \Aut(f)$.

The orbifold cohomology (or Chen--Ruan cohomology in some contexts) of $\cX$ is defined to be the cohomology of the inertia stack, with a degree shift
$$A_\orb^*(\cX):= \bigoplus_i A^{*-\age_i}(\cX_i),$$
where $\cX_i\subset I\cX$ are connected components, and $\age_i$ is the degree shifting, called the age of $\cX_i$. For more details on the orbifold cohomology, see \cite{AGV, CR}.

The inertia stack of $[\C^2/\Z_{n+1}]$ is
$$I[\C^2/\Z_{n+1}]\cong [\C^2/\Z_{n+1}] \cup B\Z_{n+1}\cup \cdots \cup B\Z_{n+1},$$
with $n$ copies of $B\Z_{n+1}$. We call $[\C^2/\Z_{n+1}]$ the \emph{untwisted} component and $B\Z_{n+1}$'s the \emph{twisted} components. The orbifold cohomology is
$$A_\orb^*([\C^2/\Z_{n+1}])= A^*([\C^2/\Z_{n+1}]) \oplus \bigoplus_{i=1}^n \C\cdot e_i,$$
where the $e_i$'s are represented by the gerby points, each with age $1$. In particular, they are divisors in $A^*_\orb$.

The (compactly supported) K-theory of $[\C^2/\Z_{n+1}]$ is equivalent to the $\Z_{n+1}$-equivariant K-theory of $\C^2$. Then
$$K([\C^2/\Z_{n+1}])\cong K_{\Z_{n+1}}(\C^2) \cong \text{Rep}(\Z_{n+1}),$$
with generators $\{\C\otimes \rho_i| 0\leq i\leq n\}$, where $\C$ stands for the structure sheaf of the origin, and $\rho_i$ is the irreducible representation of $\Z_{n+1}$ with character $\zeta\mapsto \zeta^i$.

\subsection{$\Hilb^m(\cA_n)$ and $\widehat\fgl(n+1)$-action}

Let $\cA_n$ be the minimal resolution of $\C^2/\Z_{n+1}$. The cohomology of $\Hilb^m(\cA_n)$ is well-known to be described by the Nakajima operators
$$\xymatrix{
\fp_{-k}(\gamma): A^*(\Hilb^l(\cA_n)) \ar[r] & A^*(\Hilb^{l+k}(\cA_n)),
}$$
where $\gamma\in A^*(\cA_n)$. They satisfy the Heisenberg relation:
$$[\fp_k(\alpha),\fp_l(\beta)]=-k\delta_{k+l,0} \langle \alpha,\beta \rangle c, \qquad \fp_k(\gamma)^*=(-1)^k \fp_{-k}(\gamma),$$
where we adopt the sign convention of Maulik--Oblomkov \cite{MOb09-QH}.

In particular, if we pick a basis $\cB$ for $A^*(\cA_n)$, then we have the Nakajima basis for $\Hilb^m(\cA_n)$:
$$\fp_{-\mu_1}(\gamma_1) \fp_{-\mu_2}(\gamma_2) \cdots \fp_{-\mu_l}(\gamma_l) 1,$$
and a modified version
\begin{equation} \label{Nak_basis}
\mu[\gamma]:= \frac{1}{\fz(\mu)} \cdot \fp_{-\mu_1}(\gamma_1) \fp_{-\mu_2}(\gamma_2) \cdots \fp_{-\mu_l}(\gamma_l) 1,
\end{equation}
where $\mu= (1^{m_1} 2^{m_2}\cdots)$ is a partition of $m$ of length $l$, $\gamma_i\in \cB$, and $\fz(\mu)=|\Aut(\mu)|\prod_{i=1}^l \mu_i = \prod_i i^{m_i} m_i!$.

The pairing between these classes is
$$\langle \mu[\gamma] | \nu[\gamma'] \rangle = \frac{(-1)^{m-l(\mu)}}{\fz(\eta)} \delta_{\mu\nu} \prod_{i=1}^{l(\mu)} \langle \gamma_i | \gamma'_i \rangle,$$
which also holds in the $T$-equivariant setting.

We will frequently use the following basis. Let $\{E_i | 1\leq i\leq n\}$ be the exceptional curves in $\cA_n$, viewed as curve classes in $A_1(\cA_n)$. Let $\{\omega_i | 1\leq i\leq n \}$ be their dual basis in $A^1(\cA_n)$. Let $p_1,\cdots,p_{n+1}\in \cA_n$ be the $T$-fixed points, and $\{w_i^\pm | 1\leq i\leq n+1 \}$ be their tangent weights. Then in $A_T^*(\cA_n)\otimes \Q(s_1,s_2)$ we have the decomposition
\begin{equation} \label{dec_of_1}
1= \frac{[p_1]}{w_1^- w_1^+} + \cdots + \frac{[p_{n+1}]}{w_{n+1}^- w_{n+1}^+},
\end{equation}
where
$$w_i^-= (n+2-i)s_1+ (1-i)s_2, \qquad w_i^+= (-n+i-1)s_1 + is_2,$$
and we can set
\begin{equation} \label{dec_of_pt}
[\pt]:= \frac{\sum_{i=1}^{n+1} [p_i]}{n+1}
\end{equation}
to be the dual of $1$.

If we write everything in terms of the fixed-point basis, we can see that both $$\Omega:=\{1, \omega_1, \cdots, \omega_n\}, \qquad \cE:=\{[\pt], E_1, \cdots, E_n\},$$
are bases for the equivariant cohomology $A_T^*(\cA_n)\otimes \Q(s_1,s_2)$. As a consequence, classes in the form of (\ref{Nak_basis}) where $\gamma_i$ are taken in $\Omega$ (resp. $\cE$) form a basis of $A_T^*(\Hilb^m(\cA_n))\otimes \Q(s_1,s_2)$. Moreover, for $\Omega$ (but not $\cE$), one actually obtains a basis for $A_T^*(\Hilb^m(\cA_n))$ over $\Q[s_1,s_2]$. One of the advantages of these bases that will be important for us later is that most classes here are compactly supported.

Following Maulik--Oblomkov \cite{MOb09-QH, MOb09-DT}, there is a $\widehat{\fgl}(n+1)$-action on this cohomology, constructed as follows. Let $\fg=\fgl(n+1)$ be the Lie algebra of $(n+1)\times (n+1)$ matrices. The corresponding affine Lie algebra is defined as the standard extended central extension of its loop algebra:
$$\hat{\fg}=\widehat{\fgl}(n+1):= \C[t,t^{-1}]\otimes_\C \fg \oplus \C c \oplus \C d,$$
with Lie brackets
$$\left[ t^k\otimes x, t^l\otimes y \right]:= t^{k+l}\otimes [x,y] + k\delta_{k+l,0} \tr(xy) c,$$
$$\left[ d, t^k\otimes x \right]:= kt^k\otimes x,$$
where $x,y\in \fg$, and other brackets trivial. We write $t^k\otimes x$ as $x(k)$.

The Cartan subalgebra of $\hat\fg$ is
$$\hat\fh = \fh\oplus \C c \oplus \C d,$$
with weight space
$$\hat{\fh}^* = \fh^* \oplus \C \Lambda \oplus \C \delta,$$
where $\Lambda$ and $\delta$ are defined as
$$\Lambda(\fh)= \delta(\fh)=0, \quad \Lambda(c)=\delta(d)=1, \quad \Lambda(d)=\delta(c)=0.$$
The roots of $\hat\fg$ are
$$\Phi = \{ k\delta \pm (\alpha_i +\cdots + \alpha_{j-1}) \mid k\in \Z, \ 1\leq i<j\leq n+1\} \cup \{k\delta \mid k\neq 0 \},$$
where $\{\alpha_i | 1\leq i\leq n\}$ are simple roots of $\fg$.

The Heisenberg algebra generated by $A^*_T(\cA_n)$ can be identified with the canonical Heisenberg subalgebra in $\hat\fg$:
$$\fp_{-k}(1) \mapsto \Id (-k), \qquad \fp_k(\pt) \mapsto -\frac{\Id(k)}{n+1}, \qquad k>0;$$
$$\fp_k(E_i) \mapsto e_{i,i}(k)-e_{i+1,i+1}(k), \qquad c\mapsto 1,$$
where $e_{i,j}\in \fg$ denotes the matrix whose only nonzero entry is a $1$ at the position $(i,j)$. In other words, the Heisenberg algebra is naturally embedded into $\hat\fg$ as the subalgebra $\bigoplus_{k\neq 0}t^k\otimes \fh \oplus \C c \subset \hat\fg$.

Let $V_\Lambda$ be the irreducible highest weight representation with highest weight $\Lambda$ and highest vector $v_\emptyset$, or in other words, the basic representation in the sense of Frenkel--Kac. Let $V_\Lambda[\Lambda-m\delta]$ be the weight space with weight $\Lambda-m\delta$. Then we have
$$\cF_{\cA_n}^T := \bigoplus_{m\geq 0} A_T^*(\Hilb^m(\cA_n),\Q) \otimes \Q(s_1,s_2) \cong \bigoplus_{m\geq 0} V_\Lambda[\Lambda-m\delta]\otimes \Q(s_1,s_2),$$
as graded vector spaces, respected by the Heisenberg algebra.


\subsection{Cyclic quiver varieties and $\Hilb^\bm([\C^2/\Z_{n+1}])$}

By definition, $\Hilb^\bm([\C^2/\Z_{n+1}])$ is the Hilbert scheme parameterizing certain 0-dimensional closed substacks with proper supports in $[\C^2/\Z_{n+1}]$. Here $\bm=(m_0,\cdots, m_n)$ is a tuple of nonnegative integers, which stands for
$$\sum_{j=0}^n m_j \rho_j \in K([\C^2/\Z_{n+1}]).$$
$\Hilb^\bm([\C^2/\Z_{n+1}])$ inherits a natural $T$-action from $[\C^2/\Z_{n+1}]$.

An alternative interpretation of this Hilbert scheme is the $\Z_{n+1}$-equivariant Hilbert scheme, parameterizing $\Z_{n+1}$-invariant closed subschemes with proper supports in $\C^2$, with the same topological data. This interpretation can be rephrased by the language of quiver varieties.

For any quiver $Q$ without loops, pick a dimension-framing vector $(\bv, \bw)$ and a stability condition $\theta$. One can define Nakajima's quiver variety $\fM_\theta(\bv,\bw)$, which we simply describe here. For the general theory on quiver varieties, see \cite{Nak94, Nak98}.

Let $I$ be the set of vertices in $Q$. Let $\Omega$ be the set of the edges with a certain orientation (without cycles), and $\bar \Omega$ be the one with all orientations reversed. Let $H:=\Omega\sqcup \bar\Omega$. Let $\bv= (v_i)_{i\in I}$ and $\bw=(w_i)_{i\in I}$ be the dimension vector and framing vector. We assign to each vertex the vector spaces $V_i$ and $W_i$, of dimension $\bv_i$ and $\bw_i$.

Consider cotangent bundle of the space of quiver representations
\begin{equation} \label{large-Hom}
\bigoplus_{h\in \Omega} \Hom(V_{\out(h)}, V_{\text{in}(h)}) \oplus \bigoplus_{h\in \bar\Omega} \Hom(V_{\out(h)}, V_{\text{in}(h)}) \oplus \bigoplus_{i\in I} \Hom(W_i, V_i) \oplus \bigoplus_{i\in I} \Hom(V_i, W_i),
\end{equation}
in which an element is denoted by
$$(B, i,j)= \{(B_h,i_k,j_k), \ h\in H, \ k\in I \}.$$

Let $\varepsilon$ be the map sending $h\in \Omega$ (resp. $\bar\Omega$) to $+1$ (resp. $-1$). Define the symplectic form
\begin{eqnarray*}
\omega( (B, i, j), (B', i', j')) &:=& \sum_{h\in H} \text{tr} ( \varepsilon (h) B_h B'_{\bar h} ) + \sum_{k\in I} \text{tr} (i_k j'_k - i'_k j_k) \\
&=& \sum_{h\in \Omega} \varepsilon (h) \text{tr} ( [B_h, B'_{\bar h}] ) + \sum_{k\in I} \text{tr} (i_k j'_k - i'_k j_k).
\end{eqnarray*}

Let $G$ be the group $\prod_{i\in I} GL(V_i)$, acting on the space of quiver representations by
$$g\cdot (B, i, j) := ( g_{\text{in}}(h) B_h g_{\out(h)}^{-1}, g_{\text{in}}(h) i_k, j_k g_{\out(h)}^{-1}).$$

The moment map of this action is
$$\mu_\C (B, i, j) := \left( \sum_{\text{in}(h)=k} \varepsilon(h) B_h B_{\bar h} + i_k j_k \right)_{k\in I} \in \bigoplus_{k\in I} \Hom(V_k,V_k).$$

Consider the affine variety $\mu^{-1}(0)$. Choose a \emph{stability condition} $\theta=(\theta_i)_{i\in I}$, which defines a character $\chi: G\rar \C^*$, $\chi(g)= \prod_{k\in I} (\det g_k)^{\theta_k}$. Define the quiver variety to be the GIT quotient
$$X=\fM_\theta(\bv,\bw):= \mu^{-1}(0) \sslash_\theta G.$$

There is also a hyper-K\"ahler construction of $\fM_\theta(\bv,\bw)$. The prequotient Hom space (\ref{large-Hom}) can be endowed with another complex structure $J$ and therefore admits a quaternion structure. Introduce another real moment map
$$\mu_\R(B, i, j):= \frac{i}{2} \left( \sum_{\text{in}(h)=k} \left( B_h B_h^\dagger - B_{\bar h}^\dagger B_{\bar h} \right) +  i_k i_k^\dagger - j_k^\dagger j_k \right)_k \in \bigoplus_{k\in I} \mathfrak{u}(V_k).$$
One can choose $(\zeta_\C, \zeta_\R)$ from the image of the moment map $(\mu_\C, \mu_\R)$ and form the hyper-K\"ahler quotient $\fM_{\zeta_\C, \zeta_\R}(\bv,\bw)$, which is homeomorphic to the previous GIT quotient $\fM_\theta(\bv,\bw)$.

From now on we fix the quiver $Q$ to be a cyclic quiver of $n+1$ vertices, indexed from $0$ to $n$. Pick the dimension vector $\bv= (m_0,\cdots,m_n)$ and framing $\bw=(1,0,\cdots,0)$, and the stability condition $$\theta=(1,\cdots,1).$$
The corresponding quiver variety $\fM_\theta(\bv,\bw)$ is isomorphic to $\Hilb^\bm([\C^2/\Z_{n+1}])$. For more details, see \cite{Ku, Nag}.

We are particularly interested in the multi-regular case
$$m_0= \cdots =m_n =m,$$
where $m$ is a certain nonnegative integer. In this case we simply denote the Hilbert scheme by $\Hilb^m \left( [\C^2/\Z_{n+1}] \right)$.

The Hilbert scheme of points $\Hilb^m(\cA_n)$ is isomorphic to the quiver variety of the same quiver and dimension-framing vector as above, but with a different stability condition \cite{Nag, Ku}:
$$\theta=(-n+\varepsilon,1,\cdots, 1),$$
where $\varepsilon$ is an arbitrarily small positive number. By the hyper-K\"ahler rotations, quiver varieties with different but generic stability conditions are diffeomorphic to each other.

Moreover, there is an $S^1$-action on the quiver variety, defined as
$$t \cdot \left( B_h, B_{\bar h}, i_k, j_k \right) := \left( t B_h, t^{-1} B_{\bar h}, i_k, j_k \right),$$
which is compatible with the natural anti-diagonal torus action on $\Hilb^m([\C^2/\Z_{n+1}])$ and $\Hilb^m(\cA_n)$. Note that the action makes the hyper-K\"ahler moment maps equivariant and is therefore well-defined. We have the following lemma, saying that the diffeomorphism respects the $S^1$-action.

\begin{prop} [Lemma 4.1.3 of \cite{Nag}, Corollary 47 of \cite{Ku}] \label{diff}
  There is an $S^1$-equivariant diffeomorphism
  $$\xymatrix{
  \phi: \Hilb^m([\C^2/\Z_{n+1}]) \ar[r]^-{\sim} & \Hilb^m(\cA_n).
  }$$
\end{prop}

\begin{rem}
  One may hope that the diffeomorphism respects an $S^1\times S^1$-action inherited from the $T$-actions on $\cA_n$ and $[\C^2/\Z_{n+1}]$. Unfortunately, this cannot be deduced from the hyper-K\"ahler rotation argument since the part of the torus action that scales the symplectic structure does not preserve the hyper-K\"ahler real moment map.
\end{rem}

By Proposition \ref{diff}, there is an isomorphism between the ($S^1$-equivariant, and thus $T^\pm$-equivariant) cohomologies
$$\xymatrix{
\phi^*: A^*(\Hilb^m(\cA_n)) \ar[r]^-{\sim} & A^*(\Hilb^m([\C^2/\Z_{n+1}])),
}$$
$$\xymatrix{
\phi^*: A^*_{T^\pm}(\Hilb^m(\cA_n)) \ar[r]^-{\sim} & A^*_{T^\pm}(\Hilb^m([\C^2/\Z_{n+1}])).
}$$
Hence there are also Nakajima operators and Nakajima basis on $A^*_{T^\pm}(\Hilb^m([\C^2/\Z_{n+1}]))$, given as images of those on the other side, satisfying the same Heisenberg relations, over $\Q(s)$. By abuse of notation we denote the Nakajima operators and basis in the same way as on the $\cA_n$ side.

There is also a $\widehat{\fgl}(n+1)$-action on the cohomology, realizing it as
$$\cF^{T^\pm}_{[\C^2/\Z_{n+1}]} := \bigoplus_{m\geq 0} A_{T^\pm}^*(\Hilb^m([\C^2/\Z_{n+1}]),\Q) \otimes \Q(s) \cong \bigoplus_{m\geq 0} V_\Lambda[\Lambda-m\delta]\otimes \Q(s),$$
and we also denote the full $T$-equivariant cohomology as the Fock space
$$\cF^T_{[\C^2/\Z_{n+1}]} := \bigoplus_{m\geq 0} A_T^*(\Hilb^m([\C^2/\Z_{n+1}]),\Q) \otimes \Q(s_1,s_2).$$

\begin{rem}
Nakajima \cite{Nak94, Nak98} has a general construction of actions by Kac--Moody algebras on quiver varieties. In \cite{Nag}, Nagao shows that for cyclic quivers, as in our case, this action coincides with the $\widehat{\fgl}(n+1)$-action described above, provided one reduces to the anti-diagonal torus $T^\pm\subset T$. The correspondence is given by the combinatorial $(n+1)$-quotient operations between $(n+1)$-tuples of partitions and $\Z_{n+1}$-colored partitions, which are naturally identified with $T^\pm$-fixed points of the corresponding Hilbert schemes. (In fact, Nagao only shows the correspondence for the $\widehat\fsl(n+1)$-action. But in this case it can be easily generalized to $\widehat\fgl(n+1)$.)
\end{rem}

\subsection{Divisors and effective curves}

We are particularly interested in divisors of the cyclic quiver varieties $\fM_\theta(\bv,\bw)$. The Nakajima basis for $A^1(\Hilb^m(\cA_n))$, and also $A^1_T(\Hilb^m(\cA_n))$ is
$$D_0=-\frac{1}{2 (m-2)!}\fp_{-2}(1)\fp_{-1}(1)^{m-2}1, \qquad D_i= \frac{1}{(m-1)!} \fp_{-1}(\omega_i)\fp_{-1}(1)^{m-1}1, \qquad i\geq 1.$$
There is another natural basis of divisors for quiver varieties,
$$c_1(\cV_i), \quad 0\leq i\leq n,$$
where $\cV_i$ is the tautological vector bundle associated with the $i$-th vertex.

There is a simple relation between these two bases. By Nakajima's construction of the McKay correspondence, we have $\omega_i = c_1(\cL_i)$ for $i\geq 1$, where $\cL_i$ is the $i$-th tautological line bundle on $\Hilb^1([\C^2/\Z_{n+1}])\cong \Hilb^1(\cA_n) = \cA_n$.

\begin{lem}
  For $\Hilb^m(\cA_n)$,
$$c_1(\cV_0) = D_0, \qquad c_1(\cV_i) = D_0 + D_i, \qquad i\geq 1.$$
\end{lem}

\begin{proof}
  We will apply Lehn's result on Chern classes of tautological bundles on Hilbert schemes. First we have $\cO^{[n]} \cong \cV_0$ and $\cL^{[n]} \cong \cV_i$ by the construction in Theorem 43 of \cite{Ku}.

  In Theorem 4.6 of \cite{Le}, we let $L$ be $\cO$ and take the cohomological degree one part. The only term in the exponential expansion that contributes to $A^1(\Hilb^m(\cA_n))$ is
  $$\frac{1}{(m-1)!} \left( \fp_{-1} (1) - \frac{1}{2} \fp_{-2} (1) + \cdots \right)^{m-1}  \cdot 1,$$
  with degree one part equal to $D_0$. Here $\cdots$ means terms involving $\fp_{-k} (1)$ with $k> 2$.

  For $i\geq 1$, take $L$ in the theorem to be $\cL_i$. Then the terms that contribute to $A^1(\Hilb^m(\cA_n))$ are
  $$\frac{1}{(m-1)!} \left( \fp_{-1}(1+ \omega_i) - \frac{1}{2} \fp_{-2} (1+\omega_i) + \cdots \right)^{m-1} \cdot 1 + \frac{1}{m!} \left( \fp_{-1}(1+ \omega_i) - \frac{1}{2} \fp_{-2} (1+\omega_i) + \cdots \right)^m \cdot 1,$$
  which is $D_0 + D_i$.
\end{proof}

Since the diffeomorphism $\phi$ naturally preserves tautological bundles, these equalities also hold in $\Hilb^m([\C^2/\Z_{n+1}])$, as long as one take the reduction to $T^\pm$. The decomposition (\ref{dec_of_1}) becomes
$$1= -\frac{1}{(n+1)^2s^2} \left( [p_1] + \cdots + [p_{n+1}] \right) = -\frac{1}{(n+1)s^2} [\pt].$$

For divisors and curve classes, as in \cite{MOk12}, there are identifications
$$A^1(\fM_\theta(\bv,\bw))\cong \C c\oplus \fh, \qquad A_1(\fM_\theta(\bv,\bw))\cong \C \delta\oplus \fh^*$$
in our special case. Moreover, under these identifications the root lattice of $\hat\fg$ corresponds to integral curves $H_2(\fM_\theta (\bv,\bw), \Z)$; and simple roots
$$\alpha_0, \alpha_1, \cdots, \alpha_n \in \hat\fh^*$$
correspond to the dual basis of $c_1(\cV_0), c_1(\cV_1) \cdots c_1(\cV_n)$.

By construction of quiver varieties, the divisor
$$\sum_{i=0}^n \theta_i \cdot c_1(\cV_i)$$
is ample, because it is the polarization line bundle coming from the GIT quotient procedure. Therefore the condition for a curve class to be effective is to be represented by a root $\hat\alpha \in \C\delta\oplus \fh^*$, satisfying
$$\theta\cdot \hat\alpha >0.$$
Hence the effective cone $\mathrm{Eff}$ is given by a Weyl chamber in the weight space, with polarization $\theta$. For our cyclic quiver varieties it is described as follows:
\begin{enumerate}[$\bullet$]
  \setlength{\parskip}{1ex}

  \item effective curves in $\Hilb^m(\cA_n)$ correspond to roots $k\delta + \alpha$, for all $k\in \Z$ and \emph{positive} roots $\alpha$ of $\fg$;

  \item effective curves in $\Hilb^m([\C^2/\Z_{n+1}])$ correspond to the \emph{positive} roots of $\hat\fg$, i.e. roots $k\delta +\alpha$, either $k>0$ and $\alpha$ a root of $\fg$, or $k\in \Z$ and $\alpha$ a positive root of $\fg$.
\end{enumerate}

From now on we will not distinguish between curve classes and roots of $\hat\fg$.

\subsection{Stable basis}

Another basis that would be pretty useful for us is the stable basis of Maulik--Okounkov \cite{MOk12}. It can be defined in the broader generality of symplectic resolutions. We will only treat it in our case of cyclic quiver varieties. First let's introduce the notion of Steinberg correspondence.

Let $X$ and $Y$ be (possibly non-compact) holomorphic symplectic varieties with symplectic forms $\omega_X$ and $\omega_Y$. Let $\omega_X-\omega_Y$ be the natural symplectic structure on the product $X\times Y$. Let $L\subset X\times Y$ be a cycle.

\begin{defn}
  $L$ is called a \emph{Steinberg correspondence} if
  \begin{enumerate}[1)]
  \setlength{\parskip}{1ex}
  \item each of its component is a Lagrangian subvariety in $X\times Y$;

  \item there exist proper maps $X\rar V$, $Y\rar V$ to an affine variety $V$, such that $L$ is supported on $X\times_V Y$.
  \end{enumerate}
  In cases where everything admits a group action, we require $L$ and the corresponding maps to be equivariant.
\end{defn}

Now let $X= \fM_\theta (\bv,\bw)$ be the cyclic quiver variety, and $X^{T^\pm}$ be the fixed point set. One has to choose a chamber $\fC$ in the Lie algebra of $T^\pm$, together with a polarization on $X$. For $Z\in X^{T^\pm}$, let $\text{Slope}_\fC (Z)$ be the closure of all points eventually attracted to $Z$, under the $T^\pm$ flow determined by $\fC$, and let $N_\pm$ be the subbundle of the normal bundle $N_{Z/X}$ with positive/negative weights. For precise definitions and more details of all these notions we refer to \cite{MOk12}.

\begin{thm}[Theorem 3.3.4 of \cite{MOk12}]
  There exist unique $\Q[s_1,s_2]$-linear maps
  $$\xymatrix{
  \Stab_\fC: A^*_T \left( X^{T^\pm} \right) \ar[r] & A^*_T (X),
  }$$
  such that for any $Z\in X^{T^\pm}$, and any $\gamma\in A_T^*(Z)$,
  \begin{enumerate}[(i)]
  \setlength{\parskip}{1ex}
  \item $\Stab_\fC (\gamma)$ is supported on $\operatorname{Slope}_\fC (Z)$;

  \item $\left. \Stab_\fC (\gamma) \right|_Z = \pm e(N_-) \cdot \gamma$, with signs determined by the chosen polarization;

  \item $\left. \Stab_\fC (\gamma) \right|_{Z'}$ is divisible by $(s_1+s_2)$ for all $Z'\neq Z$ that are eventually attracted to $Z$.
  \end{enumerate}
\end{thm}

In particular, the $\Stab_\fC$ map is an isomorphism of $T$-equivariant cohomologies if one passes to $\Q(s_1,s_2)$; accordingly, $\{\Stab_\fC(Z)\}$ form a $\Q(s_1,s_2)$-basis for $A^*_T(X)\otimes \Q(s_1,s_2)$, called the \emph{stable basis}.

By construction, the maps $\Stab_\fC$ are defined via Steinberg correspondences, and therefore map the middle degree to middle degree in cohomology. As a consequence, in our case of $X = \Hilb^m(\cA_n)$ or $\Hilb^m([\C^2/\Z_{n+1}])$, if $p$ is a $T^{\pm}$-fixed point, $\Stab_\fC(p)$ will live in $A^m_T(X)$.

Among plenty of pleasant properties of the stable basis, we only mention the following which will be crucial to us. Let $\Stab^\dagger_\fC$ be the adjoint to $\Stab_\fC$ with respect to the (equivariant) Poincar\'e pairing. We view them as correspondences living in $X^{T^\pm}\times X$ and $X\times X^{T^\pm}$ respectively. Let $L\subset X\times X$ be a $T$-invariant Steinberg correspondence. By definition the convolution
$$\Stab^\dagger_{-\fC} \circ L \circ \Stab_\fC,$$
is given by the push-forward of some cycle $C\subset X^{T^\pm}\times X\times X\times X^{T^\pm}$ to $X^{T^\pm}\times X^{T^\pm}$.

\begin{lem} \label{properness}
  $C$ is proper over $X^{T^\pm}\times X^{T^\pm}$.
\end{lem}

\begin{proof}
  This follows from the proof of Theorem 4.6.1 in \cite{MOk12}.
\end{proof}

\section{Relative orbifold Donaldson--Thomas theory}

\subsection{Absolute and relative DT theories in general}

In this section we start to consider the DT theory. In general, let $\cX$ be a 3-dimensional smooth quasiprojective DM stack. Let $K(\cX)$ be the (compactly supported) K-group and $F_\bullet K(\cX)$ be the topological filtration. Fix some $P\in F_1 K(\cX)$. DT theory counts curves on $\cX$ in class $P$. From now on we will use $\cM$ to denote Hilbert schemes of curves, in either absolute or relative case.

More precisely, let $\cM:=\Hilb^P(\cX)$ be the Hilbert scheme parameterizing 1-dimensional closed substacks in $\cX$ with proper supports and with the fixed K-class $P$. By \cite{Zh}, if $\cX$ is proper there is virtual fundamental class $[\cM]^\vir\in A_*(\cM)$. In the nonproper case, a perfect obstruction theory is still available. We will treat this case in the next subsection.

To define the descendent DT invariants, one can introduce a modified Chern character. Connected components of the inertia stack $I\cX$ are gerbes over their coarse moduli spaces. Given a vector bundle $V$ on $I\cX$, it splits into a direct sum of eigenbundles $\oplus_\zeta V^{(\zeta)}$ under the gerbe actions, where $V^{(\zeta)}$ has eigenvalue $\zeta$.

There is a coefficient-twisting morphism $\rho: K(I\cX)\rar K(I\cX)_\C$, defined as
$$\rho(V):= \sum_\zeta \zeta V^{(\zeta)} \in K(I\cX)_\C.$$
The modified Chern character $\ch:K(\cX)_\Q\rar A^*(I\cX)_\C$ and orbifold Chern character $\ch^\orb:K(\cX)_\Q\rar A_\orb^*(I\cX)_\C$ is defined as
$$\ch(V):= ch(\rho(\pi^*V)), \qquad \ch^\orb_k \bigg|_{\cX_i} =\ch_{k-\age_i} \bigg|_{\cX_i}, $$
where $\pi: I\cX\rar \cX$ is the usual projection and $ch$ is the usual Chern character.

We have the diagram,
$$\xymatrix{
\cM\times I\cX \ar[r]\ar[d] & I\cX \ar[d] \\
\cM\times \cX \ar[r]^-q \ar[d]_p & \cX \\
\cM
}$$
Given $\gamma\in A^l_\orb(\cX)$, define the operator
$$\xymatrix{
\ch_{k+2}^\orb(\gamma): A_*(\Hilb^P(\cX)) \ar[r] & A_{*-k+1-l}(\Hilb^P(\cX))
}$$
as
$$\ch_{k+2}^\orb(\gamma)(\xi):= p_* \left( \ch_{k+2}^\orb(\bI) \cdot I^* q^* \gamma \cap p^* \xi \right),$$
where $I: I\cX\rar I\cX$ is the canonical involution map, and $\bI$ is the universal ideal sheaf.

Given cohomology classes $\gamma_i\in A^*_\orb(\cX)$, the absolute DT invariant is defined as
$$\langle \sigma_{k_1}(\gamma_1)\cdots \sigma_{k_l}(\gamma_l) \rangle_P := \deg \left( \prod_{i=1}^l \ch^\orb_{k_i +2} (\gamma_i)\cdot [\Hilb^P(\cX)]^\vir \right),$$
where the degree map takes the degree 0 part of the homology class and pushes it forward to a point. In cases where $\cX$ is not proper but admits a torus action with proper fixed loci, the pushforward $p_*$ and the degree map can be defined via equivariant localization.

Now let's consider the relative case. Let $\cD\subset \cX$ be a (possibly disconnected) effective smooth divisor and $P\in F_1 K(\cX)$. Let
$\Delta:= \Pj_\cD (\cO_\cD \oplus N_{\cD/\cX})$ be the $\Pj^1$-bundle over $\cD$. The following modified target is called an \emph{expanded pair}.
$$\cX[k]= \cX\cup_\cD \Delta \cup_{\cD} \cdots \cup_{\cD} \Delta, \qquad k\geq 0.$$
Here each copy of $\Delta$ are glued to other components along the 0- and $\infty$-sections. We call $\cX$ the \emph{rigid} component, and $\Delta$'s the \emph{bubble} components.

The moduli space in the relative DT theory is $\cM:= \Hilb^P(\cX,\cD)$, which parameterizes 1-dimensional closed substacks $\cZ$ of K-class $P$, with proper supports, on the expanded pairs $\cX[k]$, for all possible $k$. Moreover, the substacks $\cZ$ here should satisfy the admissibility and stability conditions. Roughly speaking, admissibility means that $\cZ$ should be (in some sense) transversal to singular divisors and the distinguished divisor in $\cX[k]$, and stability means there are only finitely many automorphisms. For more details, see \cite{Zh}.

The moduli $\Hilb^P(\cX,\cD)$ is a separated DM stack of finite type, and is proper if $\cX$ itself is proper. There is also a virtual fundamental class $[\Hilb^P(\cX,\cD)]^\vir$, which comes from a perfect \emph{relative} obstruction theory over $\fA$, the classifying stack of expanded pairs with weight $P$. Let $\fX$ be the universal target associated to $(\cX, \cD)$ over $\fA$. The diagram in this case is
\begin{equation} \label{def_DT}
\xymatrix{
\cM\times_\fA (I_\fA \fX) \ar[r]\ar[d] & I_\fA \fX \ar[r]\ar[d] & I\cX \ar[d] \\
\cM\times_\fA \fX \ar[r]^-q\ar[d]_p & \fX \ar[r]\ar[d] & \cX \\
\cM \ar[r] & \fA, &
},
\end{equation}
and the obstruction theory is
$$\xymatrix{
  E^\bullet:=Rp_*(R\cH om(\bI,\bI)_0\otimes q^*\omega_{\fX/\fA})[2] \ar[r] & L^\bullet_{\cM/\fA}.
  }$$

There is an evaluation map
$$\xymatrix{
\ev: \Hilb^P(\cX,\cD) \ar[r] & \Hilb^{\cD\cdot P}(\cD),
}$$
defined by the intersection of $\cZ$ with the distinguished divisor $\cD[k]$, where $\cD\cdot P$ is the K-theoretic Gysin pull back of $P$ to $\cD$. Given $C\in A^*(\Hilb^{\cD\cdot P}(\cD))$, the relative DT invariant is defined as
$$\langle \sigma_{k_1}(\gamma_1)\cdots \sigma_{k_l}(\gamma_l) | C \rangle_P := \int_{[\Hilb^P(\cX,\cD)]^\vir} \prod_{i=1}^l \ch^\orb_{k_i +2} (\gamma_i)\cdot \ev^*C.$$
If $\cD$ is disconnected, one can put as many relative insertions as the connected components of $\cD$.

A class $P\in K(Y)$ is called \emph{multi-regular}, if it can be represented by some coherent sheaf, such that the associated representation of the stabilizer group at the generic point of each component of its support is a multiple of the regular representation. For such $P$ we have a formula for the virtual dimension in the compact case.

\begin{prop} \label{vd_general}
  Suppose $\cX$ is proper and $P$ is multi-regular. The virtual dimension of $[\Hilb^P(\cX,\cD)]^\vir$ is given by
  $$\operatorname{vdim}= -\int_\cX ch_2(I)\cdot c_1(\cX),$$
  where $[I]$ is any point in the moduli space.
\end{prop}

\begin{proof}
  Consider a representative $I$ on $\cX[k]$. Since $\dim \fA=0$, we have
  $$\mbox{vdim}= \chi(\cO,\cO)- \chi(I,I),$$
  where $\chi(E,F):= \sum (-1)^i \dim \Ext^i(E,F)$.

  First let's look at the case $k=0$. We calculate $\chi$ by the orbifold Riemann--Roch, as stated in \cite{To},
  $$\chi(I,I)= \int_{I\cX} \frac{\ch(I^\vee\otimes I)\cdot Td(I\cX)}{ch(\rho(\lambda_{-1} N^\vee))},$$
  where $N$ is the normal bundle of the local regular embedding $\pi: I\cX\rar \cX$, and the dual and tensor operations are K-theoretic. We note that
  $$\ch(I^\vee)= \overline{\ch^\vee(I)},$$
  where the bar means the complex conjugate with respect to the natural real structure $K(I\cX)\otimes_\Z \R\subset K(I\cX)\otimes_\Z \C$.

  Since $P$ is multi-regular, by d\'evissage of K-theory we always have
  $$P=[\cO_\cZ]= \sum_{\cZ_i} \text{mult}(P,\cZ_i)\cdot [\rho_{\reg,\cZ_i}]\otimes [\cO_{\cZ_i}] \mod F_0 K(\cX),$$
  where $\cZ_i$ ranges over all irreducible components of $\cZ$, $\rho_{\reg,\cZ_i}$ is the regular representation of the stabilizer group at the generic point of $\cZ_i$, and $\text{mult}(P,\cZ_i)$ is some integer. One can see that for $0\leq i\leq 2$, $\ch_i(I)$ and $\overline{\ch^\vee_i(I)}$ only depend on the $F_1K(\cX)/F_0K(\cX)$ part of $P$. Thus we conclude that
  $$\ch_0(I)=1, \qquad \ch_1(I)=0, \qquad \ch_2(I)|_{\cX_i}=0$$
  for connected components $\cX_i\subset I\cX$ except the trivial one. In fact, the twisting morphism $\rho$ acting on a multiple of $\rho_\reg$ always gives the character of the generator of the gerbe action on $\cX_i$, and thus vanishes for nontrivial $\cX_i$. Moreover,
  $$\ch_3(I)|_{\cX_j}=0$$
  for those connected components with $\dim \cX_j<3$, by dimension reasons.

  We also have $\ch_2(I)$ is real, since $\rho_\reg$ is self-dual.

  Since $N$ is trivial for 3-dimensional components, we have
  \begin{eqnarray*}
    \chi(I,I) &=& \int_{I\cX} \left( 1+ \ch_2(I)- \overline{\ch_3(I)} \right)\cdot \left( 1+ \ch_2(I)+ \ch_3(I) \right) \cdot \frac{Td(I\cX)}{ch(\rho(\lambda_{-1} N^\vee))} \\
    &=& \int_{I\cX} \frac{Td(I\cX)}{ch(\rho(\lambda_{-1} N^\vee))} + 2\int_\cX ch_2(I)\cdot Td(\cX) + \sum_{\dim \cX_j=3} \int_{\cX_j} \left( \ch_3(I)- \overline{\ch_3(I)} \right).
  \end{eqnarray*}
  The first term is exactly $\chi(\cO,\cO)$, by Riemann--Roch, and the last term vanishes since it is imaginary but $\chi$ must be an integer. We are left with
  $$\chi(\cO,\cO)-\chi(I,I)= -2\int_\cX ch_2(I)\cdot Td(\cX)= -\int_\cX ch_2(I)\cdot c_1(\cX).$$
  Now let's consider the case for general $k$. By admissibility we have
  $$\chi(I,I)= \chi(I|_\cX,I|_\cX)+ \sum_{i=1}^k \chi(I|_{\Delta_i}, I|_{\Delta_i}) - \sum_{i=0}^{k-1} \chi(I|_{\cD_i}, I|_{\cD_i}),$$
  and a similar formula for $\chi(\cO,\cO)$, where $\Delta_i, \cD_i$ are the bubble components and singular divisors of $\cX[k]$, indexed in order.

  By the results above,
  $$\chi(\cO_{\Delta_i},\cO_{\Delta_i})-\chi(I|_{\Delta_i}, I|_{\Delta_i})= -\int_{\Delta_i} ch_2(I)\cdot c_1(\Delta_i).$$
  Let $p: \Delta\rar \cD$ be the projection of $\Pj^1$-bundle. By Euler's sequence,
  $$c_1(\Delta_i)= p^*c_1(\cD)+ \cD_{i-1}+ \cD_i,$$
  and thus
  \begin{eqnarray*}
    -\int_{\Delta_i} ch_2(I)\cdot c_1(\Delta_i) &=& \int_{\Delta_i} \cZ\cdot \left( p^*c_1(\cD)+ \cD_{i-1}+ \cD_i \right) \\
    &=& \int_\cD p_*[\cZ]\cdot c_1(\cD) + \cZ\cdot \cD_{i-1}+ \cZ\cdot \cD_i \\
    &=& \cZ\cdot \cD_{i-1}+ \cZ\cdot \cD_i.
  \end{eqnarray*}
  The first term vanishes because $p_*[\cZ]$ must be 0-dimensional in $\cD$ by admissibility.

  On the other hand, one has
  $$\chi(\cO_{\cD_i},\cO_{\cD_i})-\chi(I|_{\cD_i}, I|_{\cD_i}) = \dim \Hilb^{\cZ\cdot \cD_i} (\cD_i)= 2\cZ\cdot \cD_i.$$
  Combine everything and notice the fact $\cZ\cdot \cD_i = P\cdot \cD$ for any $i$. We get the formula for the virtual dimension.
\end{proof}

There is a degeneration formula for DT theory. Let $\cX$ be the generic fiber of a simple degeneration in the sense of \cite{Zh}, which degenerates to the central fiber $\cX_0= \cX_-\cup_\cD \cX_+$.

\begin{prop}
  Let $\{C_k\}$ be a basis of the cohomology $A^*(\Hilb^{P_0}(\cD))$ for $P_0\in K_0(\cD)$ and $g_{kl}= \int C_k\cup C_l$. Then
  $$\left\langle \prod_{i=1}^r \sigma_{k_i}(\gamma_i) \right\rangle_{\cX,P} =
  \sum_{\substack{P_- + P_+ - P_0 = P, \\
                S\subset \{1,\cdots,r\}, k,l}}
  \left\langle \prod_{i\in S} \sigma_{k_i}(\gamma_i) \middle| C_k \right\rangle_{(\cX_-,\cD),P_-} g^{kl} \left\langle \prod_{i\not\in S} \sigma_{k_i}(\gamma_i) \middle| C_l \right\rangle_{(\cX_+,\cD),P_+}. $$
\end{prop}

\subsection{DT theory of $[\C^2/\Z_{n+1}]\times \Pj^1$}

In this paper the main ambient space considered is $\cX= [\C^2/\Z_{n+1}]\times \Pj^1$. This is a noncompact target, but admits a $T$-action with compact fixed locus. Let $\cY= B\Z_{n+1}\times \Pj^1\subset \cX$ be the substack. The compactly supported K-theory $K(\cX)$ is generated by classes of the form $[\rho_i\otimes \cO_\cY]$ and $[\rho_i\otimes \cO_p]$, with $0\leq i\leq n$, where $\rho_i$ are irreducible representations of $\Z_{n+1}$ and $p\in \cY$ is a point.

We fix a K-class
$$P= \sum_{j=0}^n m_j [\rho_j \otimes \cO_\cY] + \sum_{j=0}^n \varepsilon_j [\rho_j\otimes \cO_p],$$
where $m_j, \varepsilon_j\in \Z$ and $m_j\geq 0$. Denote by $\bm$, $\varepsilon$ respectively the tuples $(m_0,\cdots, m_n)$ and $(\varepsilon_0, \cdots, \varepsilon_n)$.
The relative divisor we would consider here is a disjoint union of vertical fibers in $\cX$. In other words, let $p_1, \cdots, p_r$ be points in $\Pj^1$ and let $N_i:=[\C^2/\Z_{n+1}]\times \{p_i\}$ be the corresponding fiber. A typical divisor is of the form $\coprod_i N_i$.

A perfect obstruction theory can be constructed for this noncompact target. Let $P=(\bm, \varepsilon)$ and $\cM:= \Hilb^{P}(\cX)$. Let $\bI$ be the universal ideal sheaf on $\cM\times \cX$, and $p$, $q$ be the projections. As in the proper case \cite{Zh}, we consider the map
  $$\xymatrix{
  E^\bullet:=Rp_*(R\cH om(\bI,\bI)_0\otimes q^*\omega_\cX)[2] \ar[r] & L^\bullet_{\cM},
  }$$
given by a projection of the traceless Atiyah class.

The argument of Proposition 10 in \cite{MPT} works here. Let $[\Pj^2/\Z_{n+1}]$ be the stacky quotient of the group action
$$\zeta\cdot [x:y:z]= \left[ \zeta x: \zeta^{-1} y: z \right].$$
This gives a compactification $\bar\cX:= [\Pj^2/\Z_{n+1}] \times \Pj^1$ of $\cX$.

The moduli $\cM$ can be viewed as parameterizing closed substacks on $\bar\cX$, with proper supports and contained in $\cX\subset \bar\cX$. We obtained the same moduli space and same universal family $\cZ$, but with a different universal target $\cM\times \bar\cX$. We claim that there is a duality between the two complexes
  $$Rp_*(R\cH om(\bI,\bI)_0\otimes q^*\omega_\cX)[2], \qquad Rp_*(R\cH om(\bI,\bI)_0)[1].$$
The key for this argument is that $R\cH om(\bI,\bI)_0$ is supported on the open locus $\cM\times \cX$, and the two complexes are precisely the restriction to the open locus of the complexes
  $$Rp_*(R\cH om(\bar\bI,\bar\bI)_0\otimes q^*\omega_{\bar\cX})[2], \qquad Rp_*(R\cH om(\bar\bI,\bar\bI)_0)[1],$$
where objects with bars are counterparts on $\bar\cX$ of the un-barred objects. The complexes on $\cM\times \bar\cX$ are naturally dual to each other via the Serre duality and therefore the claim is true. As a result, the tangent and obstruction space at a point $[I]\in \Hilb^P(\cX)$ can be computed by $\Ext^1(I,I)_0$ and $\Ext^2(I,I)_0$ respectively.

The target being noncompact but with a $T$-action, we can compute the virtual dimension by $T$-equivariant methods.

\begin{prop}
  For $P= (\bm, \varepsilon)$, the virtual dimension of $[\Hilb^P(\cX,\coprod N_i)]^\vir$ is $2m_0$.
\end{prop}

\begin{proof}
  The virtual dimension is still equal to $\chi(\cO, \cO-I^\vee\otimes I)$, which makes sense since $[\cO]-[I]^\vee\otimes [I]$ is still compactly supported. We can explicitly write down the classes
  $$[\cO_\cY]= (1-\rho_1)(1-\rho_1^{-1}),$$
  $$[\cO_p]= (1-\rho_1)(1-\rho_1^{-1}) [\cO_F]= (1-\rho_1)(1-\rho_1^{-1}) \left( 1- [\cO(-F)] \right),$$
  where $\rho_1$ is the irreducible representation with weight 1, and $F$ is a general vertical fiber. One can see that
  $$[\cO_\cY]^\vee= [\cO_\cY], \qquad [\cO_p]^\vee= -[\cO_p],$$
  $$P^\vee= \sum_{j=0}^n m_j [\rho_{-j}\otimes \cO_\cY] - \sum_{j=0}^n \varepsilon_j [\rho_{-j}\otimes \cO_p],$$
  where $\rho_{-j}$ is the 1-dimensional $\Z_{n+1}$-representation with eigenvalue of the generator $\zeta^{-j}$. Thus the virtual dimension is
  \begin{eqnarray*}
    \chi(\cX, P+P^\vee - P^\vee\otimes P) &=& \chi(\cY, P+P^\vee) \\
    &=& \chi \left( B\Z_{n+1}\times \Pj^1, \sum_{j=0}^n m_j (\rho_j+\rho_{-j}) \otimes \cO_{B\Z_{n+1}\times \Pj^1} \right) \\
    && +  \chi \left( B\Z_{n+1}\times \Pj^1, \sum_{j=0}^n \varepsilon_j  (\rho_j- \rho_{-j})\otimes \cO_p  \right) \\
    &=& 2m_0,
  \end{eqnarray*}
  where the derived tensor product is always taken in $\cX$ and $P^\vee\otimes P$ vanishes by dimension reasons.
\end{proof}

\begin{rem}
  In the multi-regular case $m_0=\cdots = m_n=m$, The virtual dimension can also be deduced from the general result Proposition \ref{vd_general}. By definition one can view the obstruction theory as coming from that on a compactification $\bar\cX$, where the proposition applies.
\end{rem}

There is an evaluation map
$$\xymatrix{
\ev_i: \Hilb^P(\cX,\coprod N_i) \ar[r] & \Hilb^\bm([\C^2/\Z_{n+1}]),
}$$
for each $N_i$. Given $\gamma_i\in A^*_\orb(\cX)$, $\xi_i\in A^*(\Hilb^\bm([\C^2/\Z_{n+1}]))$, we can define the relative DT invariants
$$\left\langle \sigma_{k_1}(\gamma_1)\cdots \sigma_{k_l}(\gamma_l) \middle| \xi_1,\cdots,\xi_r \right\rangle_{\bm,\varepsilon} := \int_{[\Hilb^P(\cX,\cD)^T]^\vir} \frac{\prod_{i=1}^l \ch^\orb_{k_i +2} (\gamma_i)\cdot \prod_{j=1}^r\ev_j^*\xi_j}{e_T(N^\vir)},$$
where $N^\vir$ is the virtual normal bundle of the $T$-fixed loci $\Hilb^P(\cX,\cD)^T$. Consider the generating function
$$\left\langle \sigma_{k_1}(\gamma_1)\cdots \sigma_{k_l}(\gamma_l) \middle| \xi_1,\cdots,\xi_r \right\rangle_\bm := \sum_{\varepsilon} q^{\varepsilon} \left\langle \sigma_{k_1}(\gamma_1)\cdots \sigma_{k_l}(\gamma_l) \middle| \xi_1,\cdots,\xi_r \right\rangle_{\bm,\varepsilon},$$
where $q^{\varepsilon}=\prod_{j=0}^n q_j^{\varepsilon_j}$ is a multi-index variable and we usually omit the subscript $\bm$ if it is clear from the context.

We also define the \emph{reduced} DT generating function by normalizing with the degree zero contribution,
$$\left\langle \sigma_{k_1}(\gamma_1)\cdots \sigma_{k_l}(\gamma_l) \middle| \xi_1,\cdots,\xi_r \right\rangle'_\bm:= \frac{
\left\langle \sigma_{k_1}(\gamma_1)\cdots \sigma_{k_l}(\gamma_l) \middle| \xi_1,\cdots,\xi_r \right\rangle_\bm}{\left\langle \ |\emptyset,\cdots,\emptyset \right\rangle_0},$$
where $\emptyset$ here means there is no insertion.

For simplicity, we also write
$$\langle \xi_1,\cdots,\xi_r \rangle_\bm, \qquad \langle \xi_1,\cdots,\xi_r \rangle'_\bm,$$
for the invariants with no descendants.

The invariants live in $\Q(s_1,s_2)(( q_0, \cdots, q_n ))$. A refinement of the range of allowed $q$ parameters will be given in Corollary \ref{range-q}.

For $\cX= [\C^2/\Z_{n+1}]\times \Pj^1$ the DT degeneration formula appears as a composition of operators. One can degenerate $\Pj^1$ into $\Pj^1\cup\Pj^1$, and $\cX$ degenerates into $\cX_1\cup \cX_2$. Suppose that after degeneration the relative divisors $N_1,\cdots, N_{r'}$ lie in $\cX_1$, and the others lie in $\cX_2$. We can write the degeneration formula in terms of generating functions.

\begin{prop}[Degeneration formula]
Let $\{C_a\}$ be a basis for $A^*(\Hilb^\bm([\C^2/\Z_{n+1}]))$ and $\int C_a\cup C_b = g_{ab}$.
  \begin{eqnarray*}
  \left\langle \prod_{i=1}^l \sigma_{k_i}(\gamma_i) \middle| \xi_1,\cdots,\xi_r \right\rangle_\bm &=&
  \sum_{S\subset \{1,\cdots,l\}, a,b}
        \left\langle \prod_{i\in S} \sigma_{k_i}(\gamma_i) \middle| \xi_1,\cdots,\xi_{r'}, C_a \right\rangle_\bm g^{ab} \\
        && \cdot \left\langle \prod_{i\not\in S} \sigma_{k_i}(\gamma_i) \middle| \xi_{r'+1},\cdots,\xi_{r}, C_b \right\rangle_\bm
  \end{eqnarray*}
\end{prop}

Application of the degeneration formula to degree zero partition functions shows
$$\langle \ | \emptyset,\cdots,\emptyset \rangle_0 = \langle \ | \emptyset \rangle_0^{-(k-2)},$$
where the left hand side is relative to $k$ divisors. Using this one can see that the degeneration formula also holds for reduced invariants.

\subsection{Reduced obstruction theory}

As is similar to the case of $\C^2\times \Pj^1$ and K3 fibration, we will see that the ordinary DT invariants are trivial and a reduced obstruction theory has to be constructed to define nontrivial invariants. The existence of such a reduced theory relies heavily on the holomorphic symplectic structure on the fibers.

We only consider 1-point and 2-point invariants in this subsection, i.e. the relative divisor is either $\cD= N_0\sqcup N_\infty$, with $0, \infty\in \Pj^1$, or just $\cD= N_\infty$. Throughout this subsection we assume that $\varepsilon\neq 0$. Again we will apply a similar argument as in \cite{MPT}.

As mentioned before, the obstruction sheaf of the relative DT theory is $\cE xt^2_p(\bI,\bI)_0$, where $\cE xt_p^i$ stands for the $i$-th cohomology of $Rp_*R\cH om$. The obstruction theory $Rp_* P\cH om(\bI, \bI)_0$ is perfect, and hence can be represented by a 2-term complex $[E^0 \to E^1]$, whose cohomology at $E^1$ is $\cE xt_p^2 (\bI, \bI)_0$. Apply the left exact functor $\cH om (-, \cO_\cM)$ and by Serre duality we have
\begin{equation} \label{dual}
\cE xt^1_p (\bI,\bI\otimes q^*\omega_{\fX/\fA})_0 \cong \cH om(\cE xt^2_p(\bI,\bI)_0, \cO_\cM).
\end{equation}

Let $\fP\rar \fA$ be the universal family of expanded pairs with respect to the smooth pair $(\Pj^1, \{0\}\cup \{\infty\})$ (resp. $(\Pj^1, \{\infty\})$), and note that $\fX = [\C^2/\Z_{n+1}] \times \fP$. Let $\pi_1: \fX\rar [\C^2/\Z_{n+1}]$, $\pi_2: \fX\rar \fP$ be the projections. We have
$$\Omega_{\fX/\fA} \cong \pi_1^* \Omega_{[\C^2/\Z_{n+1}]} \oplus \pi_2^* \Omega_{\fP/\fA}.$$

Consider the image $\At_{\fP/\fA}$ of relative Atiyah class under the map
$$\xymatrix{
\Ext^1 (\bI, \bI\otimes q^*L^\bullet_{\fX/\fA}) \ar[r] & \Ext^1 (\bI, \bI\otimes q^*L^\bullet_{\fX/\fA})_0 &&
}$$
$$\xymatrix{
 && \ar[r] & H^0 (\cE xt^1_p (\bI,\bI\otimes q^*\Omega_{\fX/\fA})_0) \ar[r] & H^0 (\cE xt^1_p (\bI,\bI\otimes q^*\pi_2^*\Omega_{\fP/\fA})_0),
}$$
which gives a map
$$\xymatrix{
\cO_\cM \ar[rr]^-{\At_{\fP/\fA}} && \cE xt^1_p (\bI,\bI\otimes q^*\pi_2^*\Omega_{\fP/\fA})_0.
}$$

Composing with the map $\Omega_{\fP/\fA}\rar \omega_{\fP/\fA}$ and cupping with symplectic form $\sigma$ pulled back from $[\C^2/\Z_{n+1}]$ yields
\begin{equation} \label{vector_omega}
\xymatrix{
\cO_\cM \ar[r] & \cE xt^1_p (\bI,\bI\otimes q^*\pi_2^*\omega_{\fP/\fA})_0 \ar[r]^-{\cup \sigma}_\sim & \cE xt^1_p (\bI,\bI\otimes q^*\omega_{\fX/\fA})_0.
}\end{equation}
Hence by the duality (\ref{dual}), we obtain a map
\begin{equation} \label{surj}
\xymatrix{
\cE xt^2_p (\bI,\bI)_0 \ar[r] & \cO_\cM.
}\end{equation}

\begin{lem}
  The map (\ref{surj}) is surjective.
\end{lem}

\begin{proof}
By the vanishing of higher cohomology sheaves, it suffices to show the surjectivity over a closed point $[I]$. We are reduced to the injectivity of the map
\begin{equation} \label{2_maps}
\xymatrix{
\C \ar[r]^-{\At_{\Pj^1[k]}} & \Ext^1(I,I\otimes \pi_2^*\Omega_{\Pj^1[k]})_0 \ar[r] & \Ext^1(I,I\otimes \pi_2^*\omega_{\Pj^1[k]})_0 \ar[r]^-{\cup \sigma}_\sim & \Ext^1(I, I\otimes \omega_{\cX[k]})_0,
}\end{equation}
for a certain $k$. Note that by construction of the perfect obstruction theory, these $\Ext$-groups indeed compute the fibers of $\cE xt$-sheaves at closed points.

The first arrow is given by the pull-back of the Atiyah class on $\Pj^1[k]$. To prove that it is injective, we will construct some vector field $V$ on $\Pj^1[k]$, such that its paring with the Atiyah class is nonzero.

For 2-point invariants, the relative divisor is $N_0\sqcup N_\infty$. We pick $V$ such that it vanishes on the nodal points and the $0, \infty$ of the two boundary components of $\Pj^1[k]$. For 1-point invariant, the relative divisor is $N_\infty$. We pick $V$ the same way on the bubble components, but on the rigid $\Pj^1$ let $V$ vanish at $\infty$ and another arbitrary point which avoids the support of $\varepsilon$.

Cup $V$ with the Atiyah class. The image is a class in $\Ext^1(I,I)_0$, which stands for the deformation of $I$ in the direction along $V$. Since $\varepsilon\neq 0$ and $I$ is admissible, it is a nontrivial deformation and we conclude that the first arrow of (\ref{2_maps}) is injective.

Moreover, note that we can actually require $V$ to have opposite residues at the two components at each node; in other words, require that $V\in T_{\Pj^1[k]}^{\log}\cong \omega_{\Pj^1[k]}^\vee$. For this choice of $V$ the argument above shows that the image of $\C$ in $\Ext^1(I,I\otimes \pi_2^*\omega_{\Pj^1[k]})_0$, whose paring with $V$ is nontrivial. Hence the injectivity of (\ref{2_maps}) is proved.
\end{proof}

If we are working with the absolute DT invariants, by Theorem 1.1 of \cite{KL}, the surjectivity of (\ref{surj}) is sufficient to conclude the vanishing of the invariants and the existence of a reduced obstruction theory, since the obstruction theory is already an absolute one. However, in order to define a reduced class for the relative DT theory, we have to pass from the relative obstruction theory over $\fA$ to an absolute one. The standard way is to consider the following exact triangle
$$\xymatrix{
L_\fA^\bullet \ar[r] & L_\cM^\bullet \ar[r] & L_{\cM/\fA}^\bullet \ar[r] & L_\fA^\bullet [1].
}$$
Take its composition with the original relative obstruction theory $E^\bullet\rar L^\bullet_{\cM/\fA}\rar L^\bullet_\fA [1]$, and let $\cE^\bullet:= \text{Cone}(E^\bullet\rar L^\bullet_\fA [1])[-1]$. There is a diagram of exact triangles
$$\xymatrix{
\cE^\bullet \ar[r]\ar[d] & E^\bullet \ar[r]\ar[d] & L_\fA^\bullet [1] \ar[d] \\
L_\cM^\bullet \ar[r] & L_{\cM/\fA}^\bullet \ar[r] & L_\fA^\bullet [1].
}$$
Here $\cE^\bullet$ is also 2-term and $\cE^\bullet\rar L^\bullet_\cM$ forms an absolute perfect obstruction theory, with obstruction sheaf $\Ob_\cM:= h^1((\cE^\bullet)^\vee)$.

\begin{prop}
  The surjection $\Ob_{\cM/\fA}\twoheadrightarrow \cO_\cM$ descends to the surjection $\Ob_\cM \twoheadrightarrow \cO_\cM$.
\end{prop}

\begin{proof}
  The proof is essentially the same as Proposition 14 of \cite{MPT}. One has to check the vanishing of the composition
  \begin{equation} \label{anni}
  \xymatrix{
  T_\fA \ar[r] & h^1 ((L_{\cM /\fA}^\bullet)^\vee) \ar[r] & \cE xt_p^2 (\bI,\bI)_0 \ar[r] & \cO_\cM.
  }
  \end{equation}
  The same argument as Proposition 13 of \cite{MPT} works here. (\ref{anni}) can be replaced by
  $$\xymatrix{
  T_\fA \ar[r] & R^1 p_* ((\pi_2^*L_{\fP /\fA}^\bullet)^\vee) \ar[rr]^-{\At_{\fP /\fA}} && \cE xt_p^2 (\bI,\bI)_0 \ar[r] & \cO_\cM,
  }$$
  where the first map is the Kodaira--Spencer map. By the vanishing of higher $Rp_*$ and $\cE xt$-sheaves, over a closed point the second map is
  $$\xymatrix{
  \Ext^1_{\cX[k]}(\pi_2^* \Omega_{\Pj^1[k]}, \cO) \ar[rr]^-{\At_{\fP/\fA}} & & \Ext^2(I,I)_0
  }.$$
  The group $\Ext^1_{\cX[k]}(\pi_2^* \Omega_{\Pj^1[k]}, \cO)$ is the deformation space of $\cX[k]$ coming from smoothing the nodal divisors, and cupping with the Atiyah class gives the obstruction of extending the sheaf $I$ along the direction of smoothing nodes. However, by the existence of a universal ideal sheaf on the universal family $\fX\rar \fA$, we know that along a direction of smoothing the nodes, $I$ can always be extended. Thus the map above vanishes, and the proposition is proved.
\end{proof}

The entire argument works $T$-equivariantly. Since as a $T$-equivariant sheaf $\omega_{\cX[k]}$ has an extra factor $(t_1 t_2)^{-1}$, equivariantly the surjection (\ref{surj}) is
$$\xymatrix{
\Ob_{\cM/\fA}:= \mathcal{E}xt^2_p(\bI,\bI)_0 \ar@{->>}[r] & \cO_{\cM}\otimes (t_1t_2),
}$$
and the surjection from the absolute obstruction sheaf is,
$$\xymatrix{
\Ob_{\Hilb^P(\cX, \cD)} \ar@{->>}[r] & \cO_{\Hilb^P(\cX, \cD)} \otimes (t_1 t_2).
}$$
In the language of Kiem--Li \cite{KL}, this is a \emph{cosection} which is nowhere degenerate. We conclude from \cite{KL} that the intrinsic normal cone, in the sense of \cite{Be}, lies in the kernel $\cE_1\subset \text{Ob}_{\Hilb^P(\cX,\cD)}$ of this cosection as a \emph{cycle} $C$. As usual, one defines a \emph{reduced} virtual cycle by restricting $C$ to the zero-section of $\cE_1$,
$$\left[ \Hilb^P(\cX,\cD) \right]^\red:= 0^!_{\cE_1} [C] \in A_*(\Hilb^P(\cX, \cD)).$$
The reduced class is of virtual dimension $2m_0+1$, and is related to the usual virtual cycle by
$$\left[ \Hilb^P(\cX, \cD) \right]^\vir = (s_1+s_2)\cdot \left[ \Hilb^P(\cX, \cD) \right]^\red.$$
This construction leads to the following result.

\begin{prop}
  For $\varepsilon\neq 0$, the DT invariants
  $$\left\langle \prod_i \sigma_{k_i}(\gamma_i) \right\rangle_{\bm, \varepsilon}, \qquad \left\langle \prod_i \sigma_{k_i}(\gamma_i) \middle| A \right\rangle_{\bm, \varepsilon}, \qquad \left\langle A \middle| \prod_i \sigma_{k_i}(\gamma_i) \middle| B \right\rangle_{\bm, \varepsilon}$$
  are divisible by $s_1+s_2$.
\end{prop}

\begin{proof}
  The latter two cases follow from the reduced obstruction theory. By definition, it suffices to prove that there are no $(s_1+s_2)$-factors appearing in the denominator $e_T(N^\vir)$ as we apply the $T$-localization. One can apply the same argument as in Lemma 5 of \cite{OP-DT}, push everything forward by a Hilbert-Chow morphism and conclude that the $T$-equivariant DT invariants always take values in $\Q[s_1,s_2]_{(s_1 s_2)}$.

  For the first case, consider the degeneration of $\Pj^1$ into $\Pj^1 \cup \Pj^1$ and their product with $[\C^2/\Z_{n+1}]$. The proposition follows from the degeneration formula and the results for 1-point functions.
\end{proof}

\begin{rem}
  The proposition actually holds for DT invariants with arbitrarily many insertions. One can adopt the trick in Section 4.6 and 4.7 of \cite{OP-DT}, the key for which is that 1-point invariants for $\varepsilon=0$
  $$\left\langle \prod \sigma_{k_i} (\gamma_i) \middle| A \right\rangle_{\bm,0}$$
  form an invertible matrix. Thus by the degeneration formula, DT invariants modulo $(s_1+s_2)$ with $r$ insertions can be recovered by those with $r-1$ insertions.
\end{rem}

\section{Rubber geometry}

In this section we introduce the rubber geometry for $\cX=[\C^2/\Z_{n+1}]\times \Pj^1$, which is parallel to the theory of Okounkkov--Pandharipande \cite{OP-DT} and Maulik--Oblomkov \cite{MOb09-DT}. Most constructions there are still valid in the orbifold case without modification. We fix $P$ as before and always assume that $\varepsilon\neq 0$.

Let $\Hilb^{\bm,\varepsilon}(\cX, N_0\sqcup N_\infty)^\circ\subset \Hilb^{\bm,\varepsilon}(\cX, N_0\sqcup N_\infty)$ be the open locus parameterizing substacks $\cZ$ whose restrictions to the rigid component have only finitely many automorphisms under the $\C^*$-action. The \emph{rubber} moduli space is defined as
$$\Hilb^{\bm,\varepsilon}(\cX, N_0\sqcup N_\infty)^\sim:= [\Hilb^{\bm,\varepsilon}(\cX, N_0\sqcup N_\infty)^\circ / \C^*].$$
In other words, objects in $\Hilb^{\bm,\varepsilon}(\cX, N_0\sqcup N_\infty)^\sim$ are 1-dimensional closed substacks on the bubble components of expanded pairs $\cX[k]$. In the rubber target space, there is no longer a particular rigid component. The rubber moduli space appears naturally as the boundary of the relative moduli space. Note that $\varepsilon\neq 0$ is a necessary condition for it to be nonempty.

We adopt the notation $\cM^\sim:= \Hilb^{\bm,\varepsilon}(\cX, N_0\sqcup N_\infty)^\sim$. Analogous to previous theories, $\cM^\sim$ has a perfect obstruction theory, of virtual dimension $2m_0-1$, since the Artin stack serving as a base for $\cM^\sim$ is now of dimension $-1$. A reduced obstruction theory also exists on $\cM^\sim$, related in the same way to the ordinary obstruction theory. Rubber invariants can again be defined by $T$-localization,
$$\langle A,B \rangle^\sim_{\bm,\epsilon}:= \int_{[\cM^{\sim,T}]^\vir} \frac{\ev_0^*A\cdot \ev_\infty^*B}{e(N^\vir)},$$
and we define the generating function
$$ \langle A,B \rangle^\sim_{\bm}:= 1+ \sum_{\varepsilon\neq 0} q^\varepsilon \langle A,B \rangle^\sim_{\bm,\epsilon}\in \Q(s_1,s_2) (( q_0,\cdots,q_n )).$$

Let $\pi: \fX_\cM^\sim \rar \cM^\sim$ be the universal target over the rubber moduli space. A point in $\fX_\cM^\sim$ can be viewed as a 1-dimensional substack on the rubber together with a point in the associated target, away from the singular and distinguished divisors. Hence the extra point ``rigidifies" the component it lies in, thus producing an object in the non-rubber moduli $\cM$. This process defines a rigidification map $\phi: \fX_\cM^\sim \rar \cM$.

A more precise way to think about this is as follows. Recall that $\fX\rar \fA$ is the universal target over the classifying stack of all expanded pairs. Let $\fX^\circ\subset \fX$ be the complement of all universal singular and distinguished divisors. Consider the following substack
$$\fX_\cM^\circ:= \cM^\circ\times_\fA \fX^\circ \subset \cM\times_\fA \fX \rar \cM,$$
where $\cM^\circ=\Hilb^{\bm,\varepsilon}(\cX, N_0\sqcup N_\infty)^\circ$.

An object of $\fX^\circ_\cM$ is a pair $(\cZ,p)$ where $\cZ\subset \cX[k]$ is a curve with finite automorphisms on the rigid component, and $p\in \cX[k]$ is a point, which avoids all $\cD_i$, $0\leq i\leq k$. We have the diagram
$$\xymatrix{
\fX_\cM^\sim \ar[d] & \fX_\cM^\circ \ar@{^(->}[r]\ar[d]\ar[l] & \cM\times_\fA \fX \ar[r]\ar[d] & \fX \ar[r]\ar[d] & \cX \\
\cM^\sim & \cM^\circ \ar@{^(->}[r]\ar[l] & \cM \ar[r] & \fA. &
}$$
By definition $\cM^\sim= [\cM^\circ/\C^*]$, and the rubber universal family $\fX_\cM^\sim$ is obtained as the quotient
$$\fX_\cM^\sim = [\fX_\cM^\circ/\C^*] = [\cM^\circ \times_\fA \fX^\circ / \C^*],$$
where the $\C^*$-action on the second factor stands for the action on the rigid component.

The rigidification map $\phi: \fX_\cM^\sim \rar \cM^\circ \hookrightarrow \cM$ is then defined as follows. Given a point $(\cZ, p)$ in $\cM^\circ\times_\fA \fX^\circ$, consider the component $C_p\cong [\C^2/\Z_{n+1}]\times \Pj^1$ of $\cX[k]$ where $p$ lies in, and twist $\cZ$ by the $\C^*$-action on $C_p$ that moves $p$ into $[\C^2/\Z_{n+1}]\times \{1\}\subset C_p$. The 1-dimensional substack obtained this way is defined to be the image of $(\cZ, p)$ under $\phi$.

The definition works globally. This map is well-defined since the $\C^*$-action is free and transitive on each component. $\fX_\cM^\sim$ is isomorphic to the substack $\fX_1^\circ$ of $\fX_\cM^\circ$ with the special point lying over $1\in \Pj^1$,
$$\fX_\cM^\sim \cong \cM^\circ\times_\fA \fX^\circ \times_\cX ([\C^2/\Z_{n+1}]\times \{1\}) =: \fX^\circ_1 \subset \cM\times_\fA \fX.$$
We arrive at the following diagram
$$\xymatrix{
 & \fX_\cM^\sim \ar[dl]_\pi \ar[r]^-{\alpha}_-\sim & \fX^\circ_1 \ar[dr]^\phi & \\
\cM^\sim & & &  \cM.
}$$
The perfect obstruction theory on $\cM$ restricts to a $\C^*$-equivariant obstruction theory on the open locus $\cM^\circ$, and thus descends to $\cM^\sim$. Moreover, the universal ideal sheaf and obstruction theory also restricts to $\fX_1^\circ$, since it is a section of the trivial $\C^*$-bundle. As a reselt we have the following equality between the virtual cycles
$$\pi^*[\cM^\sim]^\vir = \alpha^*\phi^*[\cM]^\vir.$$
Let $\iota: [\C^2/\Z_{n+1}]\times \{1\}\hookrightarrow \cX$ be the embedding, and $F\in A^0_\orb([\C^2/\Z_{n+1}])$ be the fundamental class of the untwisted component. In the multi-regular case, the equality between virtual cycles leads to the following identity between descendent 2-point DT invariants and rubber invariants. Here the symbol $\langle A| \sigma_k (\gamma)| B\rangle$ just means $\langle \sigma_k (\gamma)| A, B \rangle$.

\begin{prop} \label{rigidification}
  Given $\bm=(m, \cdots, m)$, $\varepsilon\neq 0$, $A, B \in A^*(\Hilb^m([\C^2/\Z_{n+1}]))$, we have
  $$\langle A | \sigma_1(\iota_*F) | B \rangle_{m,\varepsilon} = -\frac{1}{n+1} \left( \sum_{j=0}^n \varepsilon_j \right) \cdot \langle A, B \rangle^\sim_{m,\varepsilon},$$
  $$\langle A | \sigma_0(\iota_*e_i) | B \rangle_{m,\varepsilon} = -\frac{(2-\zeta^i-\zeta^{-i})}{n+1} \left( \sum_{j=0}^n \varepsilon_j \zeta^{-ij} \right) \cdot \langle A, B \rangle^\sim_{m,\varepsilon}, \quad 1\leq i\leq n.$$
\end{prop}

\begin{proof}
  By definition, as in the diagram (\ref{def_DT}), $\langle A | \sigma_1(\iota_*F) | B \rangle_{m,\varepsilon}$ is defined by capping the class
  $$\ch_3^\orb(\iota_*F) [\cM]^\vir = p_* \left( \ch^\orb_3(\bI) \cdot I^*q^*\iota_*F \cap p^*[\cM]^\vir \right)$$
  with $\ev_0^*A\cdot \ev_\infty^*B$. Consider following class obtained from intersections on $\fX^\circ_1$,
  \begin{equation} \label{phi-class}
  \phi_* \left( \ch^\orb_3(\bI_1) \cdot I^*q_1^*F \cap \phi^*[\cM]^\vir \right),
  \end{equation}
  where $\iota_1: \fX^\circ_1 \to \cM \times_\fA \fX$ is the embedding, $\bI_1:= \iota_1^* \bI$ is the universal ideal sheaf on $\fX^\circ_1$, and $q_1: \fX^\circ_1 \to [\C^2/\Z_{n+1}]$ is the projection.

  Since $\phi = p \circ \iota_1$, we have
  \begin{eqnarray*}
    \phi_* \left( \ch^\orb_3(\bI_1) \cdot I^*q_1^*F \cap \phi^*[\cM]^\vir \right) &=& p_* \iota_{1*} \left( \iota_1^* \ch^\orb_3 (\bI) \cdot I^*q_1^*F \cap \iota_1^* p^* [\cM]^\vir \right) \\
    &=& p_* \left(  \ch^\orb_3 (\bI) \cdot I^*\iota_* F \cap p^* [\cM]^\vir \right).
  \end{eqnarray*}
  Therefore it suffices to compute (\ref{phi-class}).

  Now we apply the isomorphism $\alpha: \fX_\cM^\sim \xrar{\sim} \fX_1^\circ$, under which $q_1^* F$ becomes $q^*[F\times \Pj^1]$. Capping with $\ev_0^*A\cdot \ev_\infty^*B$, we get
  \begin{eqnarray*}
    \langle A | \sigma_1(\iota_*F) | B \rangle_{m,\varepsilon} &=& \deg \left( \phi_* \left( \ch^\orb_3(\bI_1) \cdot I^*q_1^* F \cap \phi^*[\cM]^\vir \right) \cdot \ev_0^*A\cdot \ev_\infty^*B \right)\\
    &=& \deg \left( \alpha_*\phi_* \left( \ch^\orb_3(\bI_1) \cdot \alpha^*I^*q^*_1 F \cap \alpha^*\phi^*[\cM]^\vir \right) \cdot \ev_0^*A\cdot \ev_\infty^*B \right) \\
    &=& \deg \left( \pi_* \left( \ch^\orb_3(\bI) \cdot I^*q^*[F\times \Pj^1] \cap \pi^*[\cM^\sim]^\vir \right) \cdot \ev_0^*A\cdot \ev_\infty^*B \right) \\
    &=& \deg \pi_*\left( \ch^\orb_3(\bI) \cdot I^*q^*[F\times \Pj^1] \right) \cdot \int_{[\cM^\sim]^\vir} \ev_0^*A\cdot \ev_\infty^*B \\
    &=& \left( \int_{I\cX} \ch^\orb_3(I) \cdot I^*q^*[F\times \Pj^1] \right) \cdot \langle A, B \rangle^\sim_{m,\varepsilon},
  \end{eqnarray*}
  where by abuse of notation we've also used $\bI$ to denote the universal ideal sheaf on $\fX_\cM^\sim$, and in the last equality we apply a fiberwise calculation. The same argument leads to similar formulas for other divisor insertions
  $$\langle A | \sigma_0(\iota_*e_i) | B \rangle_{m,\varepsilon} = \left( \int_{I\cX} \ch^\orb_2(I) \cdot I^*q^*[e_i\times \Pj^1] \right) \cdot \langle A, B \rangle^\sim_{m,\varepsilon}.$$
  It remains to evaluate these integrals.
  \begin{eqnarray*}
    \int_{I\cX} \ch^\orb_3(I) \cdot I^*q^*[F\times \Pj^1] &=& \int_\cX \ch_3(I) \\
    &=& -\int_\cX \sum_{j=0}^n \varepsilon_j [p] \\
    &=& -\frac{\sum_{j=0}^n \varepsilon_j}{n+1}.
  \end{eqnarray*}
  The $n+1$ appears since $p$ has automorphism group $\Z_{n+1}$. For the other integrals, let $\cX_i\cong B\Z_{n+1}\times \Pj^1\subset I\cX$ be the $i$-th component. Each of them has age 1 except $i=0$.
  $$\int_{I\cX} \ch^\orb_2(I) \cdot I^*q^*[e_i\times \Pj^1] = \int_{\cX_{-i}} \ch_2^\orb(I)\cdot q^*[e_{-i}\times \Pj^1]
    = \int_{\cX_{-i}} \ch_1(I), $$
  where we view $i\in \Z_{n+1}$.

  To compute the restriction of $[I]$ on $\cX_{-i}$ we apply the K-theoretic excess intersection formula. For $[V]\in K(\cY)$, the formula says
  $$i^*i_*[V]= \lambda_{-1}N^\vee\cdot [V],$$
  where $N_{\cY/\cX}$ is the normal bundle. In our case we have $N= \rho_1+ \rho_{-1}$ and $\lambda_{-1}N^\vee= (1-\rho_1)(1-\rho_{-1}) \in K(\cY)$. Then in $K(\cX_{-i})$,
  $$[I]|_{\cX_{-i}}= \left( 1-m\rho_\reg - \sum_{j=0}^n \varepsilon_j [\rho_j\otimes \cO_p] \right) \cdot (1-\rho_{-1})(1-\rho_1),$$
  and after the coefficient-twisting morphism $\rho$,
  $$\rho\left([I]|_{\cX_{-i}} \right) = \left( 1- \sum_{j=0}^n \varepsilon_j \zeta^{-ij} [\rho_j\otimes \cO_p] \right) \cdot (1-\zeta^i \rho_{-1})(1-\zeta^{-i} \rho_1).$$
  The term involved with $\rho_\reg$ vanishes because the character of $\rho_\reg$ on a nontrivial group element always vanishes. Compute $ch_1$ and we get
  $$\int_{I\cX} \ch^\orb_2(I) \cdot I^*q^*[e_i\times \Pj^1] = -\frac{(2-\zeta^i-\zeta^{-i}) \sum_{j=0}^n \varepsilon_j \zeta^{-ij}}{n+1}.$$
  The proposition is proved.
\end{proof}

\section{Degree zero invariants}

In this section we compute the degree zero \emph{equivariant} DT invariants. Under the specialization to the CY condition, these invariants reduce to the computation of the orbifold topological DT vertex, which is already studied by Young--Bryan \cite{YB} and more generally by Bryan--Cadman--Young \cite{Br} via certain vertex operators. We will generalize their results to the equivariant vertex. In this section everything is considered as $(\C^*)^3$-equivariant, for the 3-dimensional torus.

\subsection{Toric compactification}

For the purpose of this section, we introduce another toric compactification. As a toric orbifold surface, $[\C^2/\Z_{n+1}]$ is defined by the simplicial fan in $\R^2$ generated by $(1,0)$ and $(1,n+1)$. Now we add other two rays
$(0,1)$ and $(-1,-1)$ to obtain a complete fan. Denote the compactified surface by $\bar S$, as shown in the following figure.

\begin{figure}[h]
\begin{center}
\psfrag{a1}{$(1,n+1)$}
\psfrag{a2}{$(1,0)$}
\psfrag{a3}{$(0,1)$}
\psfrag{a4}{$(-1,-1)$}
\psfrag{b1}{$(1,n+1)$}
\psfrag{b2}{$(1,0)$}
\psfrag{b3}{$(0,1)$}
\psfrag{b4}{$(-1,-1)$}
\includegraphics[scale=0.5]{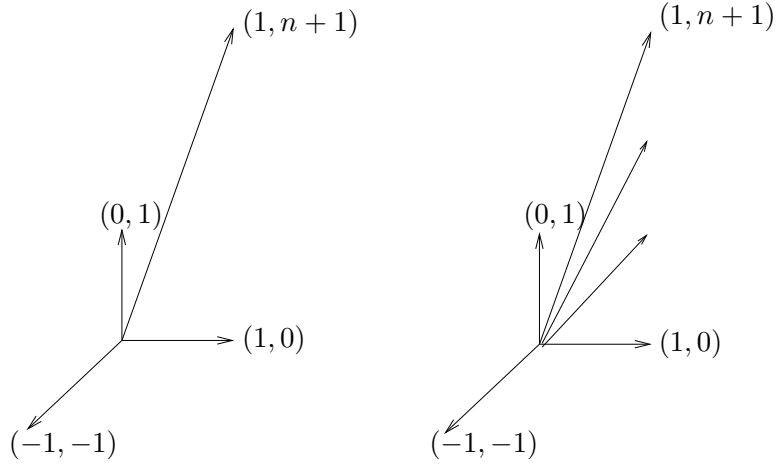}
\end{center}
\caption{toric compactifications of $[\C^2/\Z_{n+1}]$ and $\mathcal{A}_n$}
\end{figure}

A feature of this compactification which differs from $[\Pj^2/\Z_{n+1}]$ which we used before is that it does not introduce any new orbifold points: the only orbifold point in $\bar S$ is still the origin in $[\C^2/\Z_{n+1}]$. The same construction also gives a smooth toric compactification of the $\cA_n$ surface, which we denote by $\bar S'$.

Let $\partial \bar S$ and $\partial \bar S'$ be the boundaries of the two compactifications. A key observation here is that $\partial \bar S$ and $\partial \bar S'$ have exactly the same toric skeleton: same torus-fixed points, 1-skeletons and toric weights.

\subsection{Equivariant DT vertex for transversal $A_n$-singularities}

Let's compute the equivariant DT vertex for $[\C^3/\Z_{n+1}]\cong [\C^2/\Z_{n+1}]\times \C$.

Equip the lattice $\Z_{\geq 0}^3$ with the multi-regular coloring, i.e. each point $(i,j,k)$ is colored by the residue class of the content $(i-j) \mod (n+1)$. A 3d colored partition $\pi$ is viewed as a collection of boxes piled in the corner of $\Z_{\geq 0}^3$. The $\Z_{n+1}$-colored multi-regular equivariant vertex is defined as
$$Z_{\Z_{n+1}} (s,q) := \sum_\pi w(\pi) q_0^{|\pi|_0}\cdots q_n^{|\pi|_n} \quad\in\quad \Q(s_1,s_2,s_3) \llb q_0,\cdots,q_n \rrb,$$
where $|\pi|_i$ is the number of boxes in $\pi$ of color $i$, $\pi$ runs over all 3d colored partitions, and $w(\pi)$ is a $\Z_{n+1}$-version of the equivariant vertex measure in the sense of \cite{MNOP1, MNOP2}. By definition,
$$w(\pi) = \frac{1}{e\left( T_\pi^\vir \right)},$$
where $T_\pi^\vir$ is the virtual tangent space at the $(\C^*)^3$-fixed point $I_\pi$. Each $w(\pi)$ is a rational function of degree 0 in $s_1,s_2,s_3$. For an explicit formula of $T_\pi^\vir$ and $w(\pi)$, we refer to the Appendix of \cite{YB}.

\begin{lem}
  In terms of $s_3$, $\log Z_{\Z_{n+1}}(s,q)$ has only one pole of order 1.
\end{lem}

\begin{proof}
  One can use exactly the same technique as in Section 4.2 and 4.3 of \cite{MNOP2}.
\end{proof}

The subsequent argument in \cite{MNOP2} to determine the equivariant DT vertex of $\C^3$ does not work any more, since it requires the invariants to be symmetric in $s_1$, $s_2$ and $s_3$, which is not the case for $[\C^2/\Z_{n+1}]\times \C$. We look for another argument by comparison with $\cA_n\times \C$.

Denote by $Z(\cX)$ the equivariant degree 0 DT generating function of a 3-orbifold $\cX$, e.g. $Z([\C^2/\Z_{n+1}]\times \C) = Z_{\Z_{n+1}}(s,q)$. By localization we have
$$Z(\bar S\times \C) = Z([\C^2/\Z_{n+1}]\times \C) \cdot Z(\partial \bar S \times \C),$$
where by $\partial \bar S$ we denote the contribution from the fixed points on the boundary.

On the other hand, let $\bar S'$ be the smooth toric compactification of $\cA_n$ given as above. Similarly we have
$$Z(\bar S'\times \C) = Z(\cA_n\times \C) \cdot Z(\partial \bar S'\times \C),$$
as formal power series in $\Q(s_1,s_2,s_3)\llb Q \rrb$. A key observation is that $\bar S$ and $\bar S'$ has the same boundary skeleton, and therefore
$$Z(\partial \bar S \times \C) = Z(\partial \bar S' \times \C),$$
under the identification $Q = q_0 q_1 \cdots q_n$.

\begin{lem}
  $\log Z(\bar S\times \C)$ and $\log Z(\bar S'\times \C)$ only have poles in $s_3$ and have no poles in $s_1$, $s_2$.
\end{lem}

\begin{proof}
  The proof is the same as Lemma 3 in \cite{MNOP2}. Consider the proper projection $\bar S\times \C \rar \C$. One can push the virtual class forward onto $\mathrm{Sym} (\C)$, and then compute by localization on $\C$, which only introduces poles in $s_3$.
\end{proof}

Combine the previous two lemmas, the identification of boundary contributions and the fact that there is a reduced obstruction theory. We conclude that
$$\frac{Z([\C^2/\Z_{n+1}]\times \C)}{Z(\cA_n\times \C)} = F(q)^{-\frac{s_1+s_2}{s_3}},$$
where $F(q)\in \Q\llb q_0,\cdots q_n\rrb$. The denominator can computed by localization and the explicit formula for $Z(\C^3)$ in \cite{MNOP2},
$$Z(\cA_n\times \C) = M(1,-Q)^{-\frac{(n+1)(s_1+s_2)}{s_3} - \frac{(s_1+s_2)(s_1+s_2+s_3)}{(n+1)s_1s_2}}, \qquad Q = q_0q_1\cdots q_n,$$
where $M(x,Q) = \prod_{k\geq 1} \left( 1-xQ^k \right)^{-k}$ is the McMahon function.

$F(q)$ can be determined via specialization to the CY condition $s_1+s_2+s_3=0$, in which case $Z([\C^2/\Z_{n+1}]\times \C)$ reduces to the topological DT vertex in \cite{YB, Br}. By Theorem 1.6 in \cite{YB} we have
$$Z([\C^2/\Z_{n+1}]\times \C) \bigg|_{s_1+s_2+s_3=0} = M(1,-Q)^{n+1} \prod_{1\leq a\leq b\leq n} \left( M(q_{[a,b]}, Q) M(q_{[a,b]}^{-1}, Q) \right),$$
where $q_{[a,b]} = q_a\cdots q_b$, $Q=q_0 q_1\cdots q_n$. We arrive at the following conclusion.

\begin{thm}
  The $(\C^*)^3$-equivariant DT vertex for $[\C^2/\Z_{n+1}]\times \C$ is
  $$Z_{\Z_{n+1}} (s,q) = M(1,-Q)^{-\frac{(n+1)(s_1+s_2)}{s_3} - \frac{(s_1+s_2)(s_1+s_2+s_3)}{(n+1)s_1s_2}} \prod_{1\leq a\leq b\leq n} \left( M(q_{[a,b]}, Q) M(q_{[a,b]}^{-1}, Q) \right)^{-\frac{s_1+s_2}{s_3}},$$
  where $q_{[a,b]} = q_a\cdots q_b$, $Q=q_0 q_1\cdots q_n$, $M(x,Q) = \prod_{k\geq 1} \left( 1-xQ^k \right)^{-k}$.
\end{thm}

\subsection{Degree zero invariants of $[\C^2/\Z_{n+1}]\times \Pj^1$}

By relative virtual localization, the $(\C^*)^3$-equivariant 1-point degree 0 DT invariants of $[\C^2/\Z_{n+1}]\times \Pj^1$ is
$$\langle \ | \emptyset \rangle_0 =  Z_{\Z_{n+1}} (s,q) \cdot \left\langle \frac{1}{s_3-\psi_0} \middle| \emptyset \right\rangle_0^\sim,$$
where the second term is a 1-point descendent invariant on the rubber space.

\begin{thm} \label{deg-0-rel}
  The $T$-equivariant $k$-point relative degree 0 DT invariant of $[\C^2/\Z_{n+1}]\times \Pj^1$ is
  $$\langle \ | \emptyset, \cdots, \emptyset \rangle_0 = M(1,-Q)^{ (k-2)\cdot \frac{(s_1+s_2)^2}{(n+1)s_1s_2}}.$$
\end{thm}

\begin{proof}
  Denote by $W_\infty$ the rubber invariant $\left\langle \frac{1}{s_3-\psi_0} \middle| \emptyset \right\rangle_0^\sim$. By the same topological recursion argument as in Lemma 4 of \cite{MNOP2}, we have $\log W_\infty = \frac{1}{s_3} F_\infty$, where $F_\infty \in \Q(s_1, s_2) \llbracket q_0,\cdots, q_n \rrbracket$. Also by the same argument as in Lemma 3 of \cite{MNOP2}, we have $\log \langle \ | \emptyset \rangle_0$ has no poles in $s_3$ and only monomial poles in $s_1$, $s_2$. We conclude that $\log W_\infty$ consists of the terms in $-\log Z_{\Z_{n+1}} (s,q)$ with single pole $s_3$, and $\log \langle \ |\emptyset \rangle_0$ consists of the remaining terms. Hence $(\C^*)^3$-equivariantly,
  $$\log \langle \ |\emptyset \rangle_0 = -\frac{(s_1+s_2)(s_1+s_2+s_3)}{(n+1) s_1 s_2} \log M(1,-Q).$$
  The $T$-equivariant version is obtained by setting $s_3=0$, and the $k$-point invariant is obtained via the degeneration formula.
\end{proof}

Same arguments as in \cite{OP-DT} yield the following degree 0 descendent invariants. Consider the equivariant vertex with descendents
$$Z_{\Z_{n+1}}^{\sigma_1(s_3)} (s,q) := \sum_\pi w^{\sigma_1(s_3)} (\pi) q_0^{|\pi|_0}\cdots q_n^{|\pi|_n} \quad\in\quad \Q(s_1,s_2,s_3) \llb q_0,\cdots,q_n \rrb,$$
where
$$w^{\sigma_1(s_3)} (\pi) := \frac{ch_3(I_\pi)}{e\left( T_\pi^\vir \right)},$$
where $T_\pi^\vir$ is the virtual tangent space at the $(\C^*)^3$-fixed point $I_\pi$.

\begin{cor} \label{vacuum-expectation}
  \begin{enumerate}[1)]
    \setlength{\parskip}{1ex}

    \item The equivariant vertex $Z_{\Z_{n+1}}^{\sigma_1(s_3)}(s,q)$ with a descendent insertion $\sigma_1(s_3)$ is
        $$Z_{\Z_{n+1}}^{\sigma_1(s_3)}(s,q) = -s_3 Q\frac{\partial}{\partial Q} Z_{\Z_{n+1}}(s,q).$$

    \item Let $F$ be a vertical fiber of $[\C^2/\Z_{n+1}]\times \Pj^1$. Then the 2-point degree 0 DT invariant with descendent insertion $\sigma_1(F)$ is
        $$\langle \emptyset | \sigma_1(F) | \emptyset \rangle_0 = (s_1+s_2) Q\frac{\partial}{\partial Q} \log \left( M(1,-Q)^{n+1} \prod_{1\leq a\leq b\leq n} \left( M(q_{[a,b]}, Q) M(q_{[a,b]}^{-1}, Q) \right) \right).$$
  \end{enumerate}
\end{cor}

\section{Cap and tube invariants}

In this section we compute the cap and tube DT invariants, i.e. relative DT invariants with 1 or 2 relative insertions. Before getting into the calculation, we need some results from the compactification. From now on we always assume the multi-regular condition $\bm=(m,\cdots, m)$.

\subsection{Factorization on the compactification}

Again we consider the compactification introduced in Section 4. Recall that we have $S = [\C^2/\Z_{n+1}] \subset \bar S$, and $\bar\cX:= \bar S \times \Pj^1$ is a compactification of $\cX$. The torus $T$ acts fiberwisely. The $T$-fixed locus of $\bar\cX$ is a disjoint union of $B\Z_{n+1}\times \Pj^1$, which lives in the original $\cX$, and 3 copies of $\Pj^1$, which we denote by $\Pj^1(i)$, $i=1,2,3$. Each $\Pj^1(i)$ lives in a copy of $\C^2\times \Pj^1$. Let $w_1(i), w_2(i)$ be the equivariant weights of the torus action in the normal direction of $\Pj^1(i)$.

The following lemma is a product rule relating the relative DT theory of the compactification $\bar\cX$ with that of $\cX$. It is an analogue of \cite{OP-DT} (17) and \cite{MOb09-DT} Section 5.2. The summation is over all splittings of $m$ into four integers, and all splittings of $\varepsilon$ into a $\Z_{n+1}$-representation $\varepsilon(0)$ and three multiples of regular representations.

\begin{lem}[Factorization rule] \label{pre-factorization}
  Given $\bar\xi_j \in A^*_T(\Hilb^m(\bar S))$, $1\leq j\leq r$, and $\bar\gamma_i\in A^*_T(\bar\cX)$. We have the following identity,
  \begin{eqnarray*}
    \left\langle \prod_{i=1}^N \sigma_{k_i}(\bar\gamma_i) \middle| \bar\xi_1,\cdots, \bar\xi_r \right\rangle_{m,\varepsilon}^{\bar\cX} &=& \sum_{\substack{\sum_{j=1}^4 m(j)=m, \\ \varepsilon(0) + \sum_{j=1}^3 n(j) \rho_\reg = \varepsilon \\ \bigsqcup_{j=1}^4 S_j = \{1,\cdots, N\} }}  \left\langle \prod_{i\in S_0} \sigma_{k_i}(\gamma_i) \middle| \xi_1,\cdots, \xi_r \right\rangle^\cX_{m(0),\varepsilon(0)}  \\
    && \cdot \prod_{k=1}^3 \left. \left\langle \prod_{i\in S_k} \sigma_{k_i}(\gamma_i) \middle| \xi_1,\cdots, \xi_r \right\rangle^{\C^2\times \Pj^1}_{m(k), n(k)} \right|_{w_1(k),w_2(k)},
  \end{eqnarray*}
  where the insertions and relative cohomology classes on the right hand side are naturally viewed as restrictions to the corresponding open loci.
\end{lem}

\begin{proof}
  For simplicity we only give the proof in the case $r=1$ and there are no absolute insertions $\bar\gamma_i$. The proof in the general case is exactly the same. We assume the relative divisor is over $\infty\in \Pj^1$.

  Consider the generating function
  $$\langle \xi\rangle:= \sum_m \langle \xi \rangle_m y^m = \sum_{m, \varepsilon} \langle \xi \rangle_{m, \varepsilon} q^\varepsilon y^m.$$

  First we apply the $T$-localization. The fixed loci of the moduli space $\Hilb(\bar\cX, \bar S\times \{\infty\})$ consist of $T$-invariant 1-dimensional substacks supported on the disjoint union of the gerby curve $B\Z_{n+1}\times \Pj^1$ and $\Pj^1(i)$. However, the relative moduli space $(B\Z_{n+1}\times \Pj^1\sqcup \bigsqcup_{i=1}^3 \Pj^1(i), \infty\cup \infty\cup \infty\cup \infty)$ does not factor into a product of four copies.

  Next, let's apply the $(\C^*)^4$-localization, where each copy of $\C^*$ is the torus in $B\Z_{n+1}\times \Pj^1$ and $\Pj^1(i)$ respectively. By relative virtual localization, this factors the invariants into contributions from the rubber moduli space at the boundary $\infty$ and the remaining part.
  \begin{eqnarray*}
  \langle \bar\xi\rangle_{\bar\cX} &=& \sum_{\lambda_0, \lambda_1, \lambda_2, \lambda_3} \langle \emptyset \rangle^{B\Z_{n+1}\times \Pj^1\sqcup \bigsqcup_{i=1}^3 \Pj^1(i), \text{simple}}_{\lambda_0, \lambda_1, \lambda_2, \lambda_3} \cdot \prod_{j=1}^4 E(\lambda_j) \cdot \left\langle [I_{\lambda_j}] \middle| \frac{1}{-u_j-\psi_0} \middle| \xi \right\rangle^\sim \cdot \frac{1}{e(N^\vir)} \\
  &=& \sum_{\lambda_0, \lambda_1, \lambda_2, \lambda_3} \langle \emptyset \rangle^{B\Z_{n+1}\times \Pj^1, \text{simple}}_{\lambda_0} \cdot E(\lambda_0) \cdot \left\langle [I_{\lambda_0}] \middle| \frac{1}{-u_0-\psi_0} \middle| \xi \right\rangle^\sim \cdot \frac{1}{e(N_{B\Z_{n+1}\times \Pj^1}^\vir)}\\
  && \cdot \prod_{j=1}^3 \langle \emptyset \rangle^{\Pj^1(j), \text{simple}}_{\lambda_j} \cdot E(\lambda_j) \cdot \left\langle [I_{\lambda_j}] \middle| \frac{1}{-u_j-\psi_0} \middle| \xi \right\rangle^\sim \cdot \frac{1}{e(N_{\Pj^1(j)}^\vir)} \\
  &=& \sum_{\lambda_0, \lambda_1, \lambda_2, \lambda_3} \langle \emptyset \rangle^{\cX, \text{simple}}_{\lambda_0} \cdot E(\lambda_0) \left\langle [I_{\lambda_0}] \middle| \frac{1}{-u_0-\psi_0} \middle| \xi \right\rangle^\sim \\
  && \cdot \prod_{j=1}^3 \langle \emptyset \rangle_{\lambda_j}^{\C^2\times \Pj^1(j), \text{simple}} E(\lambda_j) \left\langle [I_{\lambda_j}] \middle| \frac{1}{-u_j-\psi_0} \middle| \xi \right\rangle^\sim.
  \end{eqnarray*}
  $\langle\xi\rangle^{\text{simple}}$ here stands for contributions from the rigid components, with certain prescribed asymptotes at $\infty$. $E(\lambda_j)$ are the gluing terms, and $N^\vir$ denotes the virtual normal bundles of the corresponding moduli's. $u_j$ and $\psi_0$ are the tangent weight and the psi class respectively at $0\in \Pj^1$ of the rubber. The summation is over all multi-regular 2d colored partitions $\lambda_0$, and all ordinary 2d partitions $\lambda_j$, $j=1,2,3$.

  Relative virtual localizations for $\C^2\times \Pj^1(i)$ put the invariants back into a product
  $$\langle \xi \rangle^{\bar\cX}= \langle \xi\rangle^\cX \cdot \prod_{i=1}^3 \langle \xi\rangle^{\C^2\times \Pj^1(i)},$$
  and the lemma is proved.
\end{proof}

Now we can prove the following equality between the reduced DT theories, which would be crucial for later use.

\begin{cor} \label{compact_reduced}
  Given $\xi_i\in A^*_T(\Hilb^m([\C^2/\Z_{n+1}]))$, $1\leq i\leq r$, we have
  $$\langle \xi_1, \cdots, \xi_r \rangle'_{\cX,m} = \langle \xi_1, \cdots, \xi_r \rangle'_{\bar\cX,m},$$
  where $\Hilb^m([\C^2/\Z_{n+1}])$ is naturally viewed as a open subscheme of $\Hilb^m(\bar S)$ and $\xi_i$ on the right hand side are viewed as (equivariant) pushforward's into $\bar\cX$.
\end{cor}

\begin{proof}
  The previous lemma implies that
  $$\langle \xi_1, \cdots, \xi_r \rangle_{\bar\cX,m}= \langle \xi_1, \cdots, \xi_r \rangle_{\cX,m} \cdot \prod_{i=1}^3 \langle \emptyset \rangle_{\C^2\times \Pj^1(i),0},$$
  and
  $$\langle \xi_1, \cdots, \xi_r \rangle_{\bar\cX,0}= \langle \xi_1, \cdots, \xi_r \rangle_{\cX,0} \cdot \prod_{i=1}^3 \langle \emptyset \rangle_{\C^2\times \Pj^1(i),0}.$$
  Take the quotient and we obtain the corollary.
\end{proof}

\subsection{Cap invariants}

Cap invariants, or 1-point functions, are reduced DT invariants of $\cX$ relative to one fiber $N_\infty$. In other words, they are the invariants of the form $\langle \mu[\gamma] \rangle'_{m, \varepsilon} := \langle \ | \mu[\gamma] \rangle'_{m,\varepsilon}$. We consider all $\mu[\gamma]$ in the following Nakajima basis
$$\fp_{-\mu_1}(\gamma_1) \fp_{-\mu_2}(\gamma_2) \cdots \fp_{-\mu_l}(\gamma_l)\cdot 1 \qquad \in A_{T^\pm}^{m+ l(\gamma)}(\Hilb^m([\C^2/\Z_{n+1}])),$$
where $\mu$ is a partition of $m$, $\gamma_i\in \cE:=\{[\pt], E_1, \cdots, E_n\}$, and $l(\gamma)$ is the number of $\gamma_i$ which is $[\pt]$. An important fact here is that those $\mu[\gamma]$ are compactly supported.

As a result of Corollary \ref{compact_reduced}, we have
$$\langle \mu[\gamma]\rangle'_{\cX,m,\varepsilon}= \langle \mu[\gamma]\rangle'_{\bar\cX,m,\varepsilon},$$
which lives in $\Q[s]$ due to the compactness.

Moreover, since the virtual dimension of the theory on $\bar\cX$ is $2m$, the invariants vanish for those $\mu[\gamma]$ with $m+l(\gamma)<2m$. We are left with only one nontrivial invariant, i.e.
$$(1^m)[\pt] \in A^{2m}_{T^\pm}(\Hilb^m([\C^2/\Z_{n+1}])),$$
with all $\gamma_i=[\pt]$ and $l(\gamma)=m$.

On the other hand, for $\varepsilon\neq 0$, the $T$-equivariant virtual cycle is $(s_1+ s_2)\cdot [\cM]^\red$, and the reduced cycle is of virtual dimension $2m+1$. Therefore by dimension counting again, $\langle (1^m)[\pt] \rangle'_{m,\varepsilon}$ must vanish for $\varepsilon\neq 0$. One can also see this by reduction to $s_1+s_2=0$. Hence the only case is $\varepsilon=0$, and $\langle (1^m)[\pt] \rangle'_{m,0}$ lives in $\Q$.

\begin{lem} \label{Hilb_m_0}
  For any $\bm= (m_0, \cdots, m_n)$,
  $$\Hilb^{\bm,0} \left( [\C^2/\Z_{n+1}]\times \Pj^1, \coprod N_i \right)\cong \Hilb^\bm\left([\C^2/\Z_{n+1}]\right).$$
\end{lem}

\begin{proof}
  Consider the projection $\pi_1: \cX= [\C^2/\Z_{n+1}]\times \Pj^1\rar [\C^2/\Z_{n+1}]$, which is representable. Let $\cZ\subset \cX$ be a 1-dimensional closed substack, with K-class $[\cO_\cZ]= \sum_j m_j[\rho_j \otimes \cO_\cY]$. We have the morphism defined by adjunction
  $$\pi_1^* \pi_{1*} \cO_\cZ \rar \cO_\cZ,$$
  which is surjective because $\cO_\cZ$ is globally generated. The surjectivity, together with the fact $[\pi_1^* \pi_{1*} \cO_\cZ]= \sum_j m_j [\rho_j \otimes \cO_\cY]= [\cO_\cZ]$, shows that $\pi_1^* \pi_{1*} \cO_\cZ \cong \cO_\cZ$, which implies the lemma.
\end{proof}


Applying this lemma we see that
$$\langle (1^m)[\pt]\rangle'_m = \langle (1^m)[\pt]\rangle_{m,0}
= \langle (1^m)[\pt]| (1^m)[1] \rangle_{\Hilb}= \frac{1}{m!},$$
where $\langle \ | \ \rangle_{\Hilb}$ stands for the Poincar\'e pairing on the Hilbert scheme of points.

Our conclusion for cap invariants is

\begin{thm}
  For $\gamma\in \cE:=\{[\pt], E_1, \cdots, E_n\}$, the $T^\pm$-equivariant cap invariants are
  $$\langle \mu[\gamma] \rangle'_m = \left\{ \begin{aligned}
    &\frac{1}{m!}, &&\qquad \mu[\gamma]= (1^m)[\pt], \\
    &0, &&\qquad \text{otherwise}.
  \end{aligned}\right.$$
\end{thm}

In a similar way, the full $T$-equivariant cap invariants can be determined in the $T$-fixed point basis $\{I_\mu\}$. Recall that both $T$- and $T^\pm$-fixed bases of $\Hilb^m([\C^2/\Z_{n+1}])$ are parameterized by multi-regular $(n+1)$-colored partitions $\mu$ of size $m(n+1)$.

\begin{thm}
  The $T$-equivariant cap invariants for fixed-point basis $\{I_\mu\}$ are
  $$\langle I_\mu \rangle'_m = 1, \qquad \forall \mu.$$
\end{thm}

\begin{proof}
  By the same compactness and dimension counting argument, one can observe that the only contribution to $\langle I_\mu \rangle'_m$ comes from the $\varepsilon=0$ part.
\end{proof}

The following corollary states that the fundamental class really behaves as an identity in the relative DT theory.

\begin{cor} \label{identity}
  For any $\xi_1,\cdots, \xi_r\in A^*_T \left( \Hilb^m \left( [\C^2/\Z_{n+1}] \right) \right)$,
  $$\left\langle \prod \sigma_i (\gamma_i) \middle| \xi_1, \cdots, \xi_r, (1^m)[1] \right\rangle'_m= \left\langle \prod \sigma_i (\gamma_i) \middle| \xi_1, \cdots, \xi_r \right\rangle'_m.$$
\end{cor}

\begin{proof}
  Consider the invariants on the right hand side. Degenerate $[\C^2/\Z_{n+1}]\times \Pj^1$ into $[\C^2/\Z_{n+1}]\times \Pj^1\cup [\C^2/\Z_{n+1}]\times \Pj^1$, where all descendents and relative insertions $\xi_i$'s remain on the first copy of $[\C^2/\Z_{n+1}]\times \Pj^1$. By the degeneration formula,
  $$\left\langle \prod \sigma_i (\gamma_i) \middle| \xi_1, \cdots, \xi_r \right\rangle'_m = \sum_{a,b} \left\langle \prod \sigma_i (\gamma_i) \middle| \xi_1, \cdots, \xi_r, C_a \right\rangle'_m g^{ab} \langle C_b | \ \rangle,$$
  where $\{C_a\}$ is the Nakajima basis. The results on 1-point functions imply that the only nontrivial possibility for $C_b$ is $(1^m)[\pt]$. The corollary is proved.
\end{proof}

\subsection{Tube invariants}

By tube invariants, we mean 2-point functions $\langle C_a, C_b\rangle'_m$, where $\{C_a\}$ is a basis for $A^*(\Hilb^m([\C^2/\Z_{n+1}])$. Here we need a result on the allowed $q$-parameters.

We claim that all $\langle C_a, C_b\rangle'_m$ take values in $\Q(s_1,s_2)\llb q^{\mathrm{Eff}} \rrb$, where $\mathrm{Eff}$ is a certain cone in the lattice spanned by $\varepsilon_0, \cdots, \varepsilon_n$. This claim will be proved in Corollary \ref{range-q}, and the precise meaning of the cone $\mathrm{Eff}$ will be clear there.

Consider the degeneration formula.
$$ \langle C_a, C_b \rangle'_m= \sum_{k,l} \langle C_a,C_k\rangle'_m \cdot g^{kl} \cdot \langle C_l, C_b\rangle'_m.$$
Define matrix $M$ as $M_a^b:= \sum_k \langle C_a, C_k \rangle'_m \cdot g^{kb}$. Under our assumption, the matrix has entries in $\Q(s_1,s_2)\llb q^\mathrm{Eff} \rrb$. Then the degeneration formula is exactly
$$M= M\cdot M.$$
On the other hand, by Lemma \ref{Hilb_m_0} we have
$$M|_{q=0}= \Id,$$
where taking $q=0$ means evaluating at the vertex of $\mathrm{Eff}$. Hence $M$ must be $\Id$ itself.

\begin{thm} \label{tube}
  $$\langle A,B \rangle'_m= \langle A| B\rangle_{\Hilb^m}.$$
\end{thm}

\begin{rem} \label{general_tube}
  One may observe that the computation does not involve an explicit choice of the basis $\{C_a\}$, and thus can be largely generalized. For any orbifold surface $\mathcal{S}$, we can always draw the same conclusion for the relative DT invariants of $(\mathcal{S}\times \Pj^1, \mathcal{S}_0 \sqcup \mathcal{S}_\infty)$.
\end{rem}

\section{Three-point functions}

This section is devoted to the computation of relative DT invariants with 3 relative insertions, which is the most nontrivial part of the calculation. One has to reduce the main part of the theory to quantum multiplication by divisors on $\Hilb^m \left( [\C^2/\Z_{n+1}] \right)$. As an intermediate process we have to pass to the rubber invariants. In this section $\lambda, \mu, \cdots$ will denote colored partitions, where a colored partitions is nothing but an ordinary partition with each box $(i,j)$ colored by its diagonal position $j-i \in \Z_{n+1}$.

\subsection{Localization of rubber invariants}

$T$-fixed points of $\Hilb^m([\C^2/\Z_{n+1}])$ are parameterized by $(n+1)$-colored multi-regular 2d partitions of size $m(n+1)$. They form the fixed-point basis $\{[I_\lambda]\}$ of the localized equivariant cohomology $A_T(\Hilb^m([\C^2/\Z_{n+1}])) \otimes \Q(s_1,s_2)$. Let's consider the rubber invariants under this basis
$$\langle I_\mu, I_\nu \rangle_{m, \varepsilon}^\sim, \qquad \varepsilon\neq 0.$$

Let $I$ be the representative of a $T$-fixed point in the rubber moduli space $\Hilb^{m, \varepsilon} (\cX, N_0\sqcup N_\infty)^\sim$. The target space in which $I$ lives is a chain of $[\C^2/\Z_{n+1}]\times \Pj^1$. Following \cite{OP-DT}, on each component of this chain, two possible types of $I$ can occur. Let $\cZ$ denote the associated 1-dimensional closed substack.

\begin{enumerate}[1)]
\setlength{\parskip}{1ex}

\item If there are embedded points in $\cZ$, $I$ is called a \emph{skewer}. Since the embedded points are not fixed by the $\C^*$-action, nor is the skewer. Thus it is also $T$-fixed as a point in the rigid moduli space $\Hilb([\C^2/\Z_{n+1}]\times \Pj^1)$.

    \noindent
    For skewers, one must have $I|_0= I|_\infty$ as fixed points of $\Hilb([\C^2/\Z_{n+1}])$. \\

\item If there are no embedded points in $\cZ$, $I$ is called a \emph{twistor}. In this case $\cZ$ is flat over the chain of $\Pj^1$, which by definition defines a map $f$ from the chain of $\Pj^1$ to $\Hilb([\C^2/\Z_{n+1}])$. In other words, we obtain an element
    $$[f] \in \bar\cM_{0,2}(\Hilb([\C^2/\Z_{n+1}]), \beta)^T,$$
    where the right hand side is Kontsevich's moduli space of stable maps.

    \noindent
    Let's determine the curve class $\beta$. Recall that a basis of $A_1(\Hilb([\C^2/\Z_{n+1}])$ is given by the dual basis of the divisors
    $$c_1(\cV_i)\in A^1(\Hilb([\C^2/\Z_{n+1}])),$$
    where $\cV_i$ are the tautological bundles of the quiver variety, and is identified with the simple roots of $\hat\fg=\widehat\fgl (n+1)$,
    $$\alpha_0,\alpha_1, \cdots,\alpha_n\in \hat\fh^*.$$
    Hence it suffices to compute $c_1(\cV_i)\cdot \beta$, which is the same as $\deg f^*\cV_i$.

    \noindent
    Let's view $\cZ$ as a $\Z_{n+1}$-equivariant subscheme on $\C^2\times \Pj^1$. Let $\tilde\pi_2: \C^2\times \Pj^1 \rar \Pj^1$ be the projection. By definition we have
    $$\bigoplus_{i=0}^n f^*\cV_i = \tilde\pi_{2*} \cO_\cZ$$
    as $\Z_{n+1}$-equivariant vector bundles.

    \noindent
    We compute the pushforward to $K_{\Z_{n+1}}(\pt)$.
    \begin{eqnarray*}
      \chi \left( \Pj^1, \bigoplus_{i=0}^n f^*\cV_i \right) &=& \chi \left( \C^2\times \Pj^1, \cO_\cZ \right) \\
      &=& m \rho_\reg \cdot \chi(\Pj^1) + \sum_{j=0}^n \varepsilon_j \rho_j \cdot \chi(\pt) \\
      &=& \sum_{j=0}^n (m+\varepsilon_j) \rho_j
    \end{eqnarray*}
    Hence by Riemann--Roch we have
    $$\deg f^*(\cV_i)= \chi(\Pj^1, f^*\cV_i) - \text{rk} f^*(\cV_i) = \varepsilon_i.$$
    The conclusion is
    $$c_1(\cV_i)\cdot \beta = \varepsilon_i, \qquad [f] \in \bar\cM_{0,2}(\Hilb([\C^2/\Z_{n+1}]), \varepsilon)^T.$$
\end{enumerate}

The characterization of $\beta$ shows that one can identify $\varepsilon$ with roots of $\hat\fg$ in the standard way. In other words, we identify
$$(1, 0, \cdots, 0), \quad (0, 1, 0, \cdots, 0), \quad \cdots, \quad (0, \cdots, 0, 1)$$
with $\alpha_0, \alpha_1, \cdots, \alpha_n$, or equivalently, $\varepsilon \longleftrightarrow \varepsilon\cdot \hat\alpha$.

The argument above also works for a general $\bm$. As a corollary we have a restriction on which $\varepsilon$'s are allowed to appear.

\begin{cor} \label{range-q}
  The DT invariants $\left\langle \sigma_{k_1}(\gamma_1)\cdots \sigma_{k_l}(\gamma_l) \middle| \xi_1,\cdots,\xi_r \right\rangle_\bm$ live in $\Q(s_1,s_2)\llb q^\mathrm{Eff} \rrb$, where $\mathrm{Eff}$ is the effective cone of $\Hilb^\bm([\C^2/\Z_{n+1}])$. Combined with the description of the effective cone, they take values in $\Q(s_1,s_2)\llb q_0, \cdots, q_n \rrb$.
\end{cor}

The following result relates rubber invariants (in the $\mu\neq \nu$ case) to quantum multiplication by divisors in $\Hilb^m([\C^2/\Z_{n+1}])$.

\begin{prop} \label{rubber_QH}
  For $\varepsilon\neq 0$, $\mu\neq \nu$,
  $$\langle I_\mu, I_\nu \rangle_{m, \varepsilon}^\sim = \langle I_\mu, I_\nu \rangle_{\Hilb^m, \varepsilon} \mod (s_1+s_2)^2,$$
  where the right hand side is the $T$-equivariant Gromov--Witten invariant in $\Hilb([\C^2/\Z_{n+1}])$ with $g=0$, $n=2$, and $\beta= \varepsilon$ in the sense of the above argument.
\end{prop}

\begin{proof}
  Let's apply $T$-localization to the rubber invariant $\langle I_\mu, I_\nu \rangle_{m, \varepsilon}^\sim$. Recall that a fixed point $I$ lives in a chain of rational components, and determines a graph $\Gamma_{\mu,\nu}$ in the sense of Section 8.3.5 in \cite{OP-DT}. In these graphs, vertices stand for maximal subchains of skewer components and ``breaking" vertices between twistor components; and edges stand for ``unbroken" chains of twistor components.

  By the reduced obstruction theory, if there are at least 2 edges in the chain, there would be an $(s_1+s_2)^2$ factor which does not contribute to the invariant. Therefore the only contribution comes from the graph with a single twistor edge.

  Consider the open locus $U_{m,\varepsilon}\subset \bar\cM_{0,2}(\Hilb([\C^2/\Z_{n+1}]), \varepsilon)$, consisting of those stable maps whose domains are a chain of rational curves. Then $U_{m,\varepsilon}$ can be embedded as an open locus in the rubber moduli space $\Hilb^{m,\varepsilon}(\cX, N_0\sqcup N_\infty)^\sim$, which under $T$-localization, matches exactly with the graphs with one single twistor edge. One can check, just as in Lemma 25 of \cite{OP-DT}, the perfect obstruction theory for the GW and DT theory coincide on the open locus. Moreover, the universal family and evaluation maps on both sides also coincide on $U_{m, \varepsilon}$. Thus the lemma is proved.
\end{proof}

\subsection{Correspondence between $T$-fixed points and non-punctual invariants}

Now let's deal with the case $\mu=\nu$. Before that we need a closer look at the $T$-fixed points of $\Hilb^m(\cA_n)$ and $\Hilb^m([\C^2/\Z_{n+1}])$.

Let $p_1, \cdots, p_{n+1}$ be the $T$-fixed points in $\cA_n$. $T$-fixed points in $\Hilb^m(\cA_n)$ are parameterized by $(n+1)$-tuples of partitions $$\vec{\lambda}:= \left( \lambda^{(1)}, \cdots, \lambda^{(n+1)} \right),$$
with total size $|\vec{\lambda}|:= \sum \left| \lambda^{(i)} \right| = m$, where each partition $\lambda^{(i)}$ stands for a $T$-fixed point in $\Hilb^{|\lambda^{(i)}|}(\C^2)$, concentrated at $p_i$.

It is well-known \cite{Nak96} that the cohomology of $\Hilb(\C^2)$ can be identified (as a vector space) with the ring of symmetric functions. Let $w^\pm$ be the tangent weights of the coordinate axis in $\C^2$, and $J_\lambda$ be a fixed point corresponding to a partition $\lambda$. Then the normalized fixed-point class $(w^+)^{-|\lambda|} [J_\lambda]$ is identified the integral Jack polynomial $J^{w^-/w^+}_\lambda(z)$. If we specialize to the condition $s_1+s_2=0$, these Jack polynomials specialize to Schur polynomials, and $[J_\lambda]$ can be identified with
$$(-w^+)^{|\lambda|} \frac{|\lambda|}{\dim \lambda} s_\lambda(z).$$
In particular, the Nakajima creators $\fp_{-k}([0])$ are identified with multiplications by Newton's power functions
$$w^+\cdot p_k(z).$$
Therefore, a $T$-fixed point in $\Hilb^m(\cA_n)$ can be identified with an $(n+1)$-tuple of symmetric functions, each corresponding to a fixed point in $\Hilb(\C^2_{p_i})$, where $\C^2_{p_i}$ stands for an analytic or formal neighborhood around $p_i$.

On the other hand, $T$-fixed points of $\Hilb^m([\C^2/\Z_{n+1}])$ are parameterized by multi-regular $(n+1)$-colored partitions of size $m(n+1)$. Moreover, on both $\Hilb^m(\cA_n)$ and $\Hilb^m([\C^2/\Z_{n+1}])$, $T$-fixed points coincide with $T^\pm$-fixed points, basically because there are no tangent weights divisible by $(s_1+s_2)$.

There is a perfect correspondence between the sets
$$\{ (n+1) \text{-tuples of partitions } \{\vec{\lambda}\} \text{ with total size } m \}$$
and
$$\{ \text{multi-regular } \Z_{n+1} \text{-colored partitions of size } m(n+1) \}.$$
This operation is called the $n$-\emph{quotient}. Some good introductions can be found in \cite{Br, Nag, Ross-DT}. Both types of partitions can be identified with states in certain colored-version of infinite wedge Fock space, where Nakajima operators can be realized as annihilation and creation operators.

\begin{prop}[Theorem 4.5 and 5.5 of \cite{Nag}]
  The correspondence between $T^\pm$-fixed points of $\Hilb^m(\cA_n)$ and $\Hilb^m([\C^2/\Z_{n+1}])$ induced by the $S^1$-diffeomorphism $\phi$ coincides with the $n$-quotient correspondence between partitions. Moreover, this correspondence respects the actions by Nakajima operators.
\end{prop}


One feature of this correspondence is crucial to us: one can explicitly describe how $\widehat{\fgl}(n+1)$ acts on partitions. In particular, we will use the following fact, where partitions are viewed as piles of boxes on the plane. For a reference, see Section 3.2 of \cite{Ross-DT}.

\begin{cor}
  Under the correspondence in the previous proposition, the operation of adding a single box at the position $(i,j)$ to the $l$-th partition $\lambda^{(l)}$ on the $\cA_n$ side corresponds to adding to the colored partition a length-$n$ border strip starting at a box of color $l$ on the orbifold side, where the unique color $0$ box of the length-$n$ border strip lies in the $(n+1)(j-i)$ diagonal.
\end{cor}

Let $\Theta$ be the operator on the Fock space $\cF^T_{[\C^2/\Z_{n+1}]}$ defined as
$$\langle A| \Theta| B\rangle:= \langle A, B\rangle^\sim.$$
Consider the decomposition
$$\langle A, B\rangle^\sim = 1 + \langle A, B\rangle^\sim_{\mathbf{0}} + \langle A, B\rangle^\sim_+,$$
where the ``$\mathbf{0}$" part consists of terms $q^\varepsilon$ with
$$\varepsilon_0 = \varepsilon_1 = \cdots = \varepsilon_n,$$
and the ``+" part contains other terms. If we adopt the identification $\varepsilon \longleftrightarrow \varepsilon\cdot \hat\alpha= k\delta+ \alpha$, where $\alpha$ is a root of $\fgl(n+1)$, then the ``$\mathbf{0}$" part contains exactly those terms with $q^{k\delta}$, $k\leq 1$. Similarly we define operators $\Theta_\mathbf{0}$ and $\Theta_+$, and call them ``punctual" and ``non-punctual" part respectively. Note that $\Theta_+$ is only nontrivial for $n\geq 1$.

Since for $\varepsilon\neq 0$ the invariants are divisible by $(s_1+s_2)$, we set
$$\langle A| \Theta^\red_\mathbf{0}| B\rangle:= \left. \frac{\langle A, B\rangle^\sim_\mathbf{0}}{s_1+s_2} \right|_{s_1=s, s_2=-s}, \qquad \langle A| \Theta^\red_+| B\rangle:= \left. \frac{\langle A, B\rangle^\sim_+}{s_1+s_2} \right|_{s_1=s, s_2=-s}.$$
Then $\Theta^\red_\mathbf{0}$, $\Theta^\red_+$ are operators on $\cF^{T^\pm}_{[\C^2/\Z_{n+1}]}\cong \cF^{T^\pm}_{\cA_n}$. Here ``red" stands for invariants defined by capping with the reduced virtual cycle $[\cM^\sim]^\red$, which should not be confused with the reduced DT invariants. We write $\Theta^\red:= \Theta_+^\red + \Theta^\red_\mathbf{0}$.

\begin{lem}[Parallel to Proposition 3.2 of \cite{MOb09-QH}, Proposition 4.1 of \cite{MOb09-DT}] \label{factorization}
  Let $\mu, \nu, \lambda^i, \rho^i$ be partitions. Then
  \begin{eqnarray*}
  \left\langle \mu[1] \prod \lambda^i (\omega_i) \middle| \Theta_+^\red \middle| \nu[1] \prod \rho^i (\omega_i) \right\rangle &=& \left\langle \prod \lambda^i (\omega_i) \middle| \Theta_+^\red \middle| \prod \rho^i (\omega_i) \right\rangle \langle \mu[1] | \nu[1] \rangle,\\
  \left\langle \mu[1] \prod \lambda^i (\omega_i) \middle| \Theta_\mathbf{0}^\red \middle| \nu[1] \prod \rho^i (\omega_i) \right\rangle &=& \left\langle \prod \lambda^i (\omega_i) \middle| \Theta^\red_\mathbf{0} \middle| \prod \rho^i (\omega_i) \right\rangle \langle \mu[1] | \nu[1] \rangle \\
  && + \left\langle \prod \lambda^i (\omega_i) \middle| \prod \rho^i (\omega_i) \right\rangle \langle \mu[1] | \Theta^\red_\mathbf{0} | \nu[1] \rangle \\
  && - \left\langle \mu[1] \prod \lambda^i (\omega_i) \middle| \nu[1] \prod \rho^i (\omega_i) \right\rangle \langle \emptyset | \Theta^\red_\mathbf{0} | \emptyset \rangle.
  \end{eqnarray*}
\end{lem}

\begin{proof}
  The proof of the $\Theta_+^\red$ part is exactly the same as Section 5.2 of \cite{MOb09-DT}. One works on the compactification $\bar\cX$, rigidifies by $\sigma_0(\iota_* e_i)$ and apply the product rule to both sides of the lemma. An induction argument on the total length would prove the lemma.

  It remains to prove the punctual case. Again by rigidification (Proposition \ref{rigidification}), $\Theta_\mathbf{0}$ is completely determined by the descendent invariants with insertion $\sigma_1(\iota_* F)$.

  We prove by induction on $l= \min \{ l(\mu), l(\nu) \}$. When $\mu=\nu=\emptyset$ the identity holds trivially. Now we assume either $\mu$ or $\nu$ is nonempty. Consider the compactification $\bar\cX = \bar S\times \Pj^1$. The invariants
  $$\left\langle \mu[1] \prod \lambda^i (\omega_i) \middle| \sigma_1(\iota_* F) \middle| \nu[1] \prod \rho^i (\omega_i) \right\rangle^{\bar\cX}$$
  vanish by compactness and dimension counting. Thus by the factorization rule (Lemma \ref{pre-factorization}) we have
  \begin{eqnarray*}
    0 &=& \left\langle \mu[1] \prod \lambda^i (\omega_i) \middle| \sigma_1(\iota_* F) \middle| \nu[1] \prod \rho^i (\omega_i) \right\rangle^\cX \\
    && + \sum_{j=1}^3 \left\langle \mu[1] \prod \lambda^i (\omega_i) \middle| \nu[1] \prod \rho^i (\omega_i) \right\rangle^\cX \langle \emptyset | \sigma_1(\iota_* F) | \emptyset \rangle^{\C^2\times \Pj^1(j)} \\
    && + \sum_{\mu^{(j)}, \nu^{(j)}} C_{\mu^{(j)}, \nu^{(j)}} \left\langle \mu^{(0)}[1] \prod \lambda^i (\omega_i) \middle| \sigma_1(\iota_* F) \middle| \nu^{(0)}[1] \prod \rho^i (\omega_i) \right\rangle^\cX \cdot \prod_{j=1}^3 \left\langle \mu^{(j)}[1] \middle| \nu^{(j)}[1] \right\rangle^{\C^2\times \Pj^1(j)} \\
    && + \sum_{\mu^{(j)}, \nu^{(j)}} C_{\mu^{(j)}, \nu^{(j)}} \left\langle \mu^{(0)}[1] \prod \lambda^i (\omega_i) \middle| \nu^{(0)}[1] \prod \rho^i (\omega_i) \right\rangle^\cX \cdot \sum_{j=1}^3 \left\langle \mu^{(j)}[1] \middle| \sigma_1(\iota_* F) \middle| \nu^{(j)}[1] \right\rangle^{\C^2\times \Pj^1(j)} \\
    && \cdot \prod_{j'\neq j} \left\langle \mu^{(j')}[1] \middle| \nu^{(j')}[1] \right\rangle^{\C^2\times \Pj^1(j')},
  \end{eqnarray*}
  where the summation is over all splittings of the partition $\mu$, $\nu$ of the form
  $$\mu = \mu^{(0)} \cup \bigcup_{j=1}^3 \mu^{(j)},$$
  and $C_{\mu^{(j)}, \nu^{(j)}}$ are factors coming from reordering the parts.

  When $l=0$ the last two summation terms vanish and the lemma is true. For $l>0$, the lemma follows from the induction hypothesis and the factorization rule (Lemma \ref{pre-factorization}) for $\left\langle \mu[1] \middle| \sigma_1(\iota_* F) \middle| \nu[1] \right\rangle^{\bar\cX}$.
\end{proof}

\begin{cor}
  Let $\mu, \nu$ be multi-regular $(n+1)$-colored partitions of size $m(n+1)$. Then
  $$\langle I_\mu | \Theta^\red_+ | \fp_{-k}(1) I_\nu \rangle_m, \qquad k>0$$
  is a $\Q(s)$-linear combination of rubber invariants for $m'<m$.
\end{cor}

\begin{proof}
  Write $I_\mu$, $I_\nu$ in terms of Nakajima basis. One simply observes that the coefficients do not contain $(s_1+s_2)$-factors in the denominator.
\end{proof}

\begin{prop} \label{munu-induction}
  Let $\eta$ be a multi-regular $(n+1)$-colored partition of size $m(n+1)$, $m\geq 1$. Then the invariant $\langle I_\eta | \Theta^\red_+ | I_\eta \rangle_m$ is determined by invariants of the form
  $$\langle I_\mu | \Theta^\red_+ | I_\nu \rangle_m,$$
  where $\mu, \nu$ are colored partitions with $\mu\neq \nu$, and rubber invariants with $m'<m$.
\end{prop}

\begin{proof}
  Recall that we have to assume $n\geq 1$. As long as $m\geq 1$, there always exists a multi-regular $(n+1)$-colored partition $\eta'$, of size $(m-1)(n+1)$, such that $\eta'$ is obtained by removing a length-$(n+1)$ border strip from $\eta$. By the previous lemma, we know that
  $$\langle I_\eta | \Theta^\red_+ | \fp_{-1}(1) I_{\eta'} \rangle$$
  is determined by rubber invariants with $m'<m$.

  Let's investigate the class $\fp_{-1}(1) I_{\eta'}$. One has
  $$\fp_{-1}(1)= -\frac{1}{(n+1)^2 s^2}\sum_{i=1}^{n+1} \fp_{-1}([p_i]),$$
  and each $\fp_{-1}([p_i])$ acts on an $(n+1)$-tuple of partitions by adding a box on the $i$-th partition. Under the correspondence of $(n+1)$-tuples of partitions and colored partitions, on the other side $\fp_{-1}(p_i)$ acts by adding a length-$(n+1)$ border strip at the color $i$.

  Therefore, we conclude that $\fp_{-1}(1) I_{\eta'}$ is a linear combination of $I_{\eta''}$ where $\eta''$ is obtained by adding a border strip to $\eta'$. Hence
  $$\langle I_\eta | \Theta^\red_+ | \fp_{-1}(1) I_{\eta'} \rangle
    = C_\eta\cdot \langle I_\eta | \Theta^\red_+ | I_\eta \rangle + \sum_{\eta'',\ \eta''\neq \eta} C_{\eta''}\cdot \langle I_\eta | \Theta^\red_+ | I_{\eta''} \rangle$$
  is a linear combination of rubber invariants with $m'<m$, where $C_\eta$, $C_{\eta''}$ are combinatorial coefficients. Note that $C_\eta\neq 0$. The proposition is proved.
\end{proof}

Combine this proposition with Proposition \ref{rubber_QH} we have the following result.

\begin{prop} \label{rubber_QH_red}
  For $\varepsilon\neq k\delta$, $k\in \Z$ and any multi-regular $(n+1)$-colored partitions $\mu, \nu$ (possibly the same),
  $$\langle I_\mu |\Theta^\red_+ | I_\nu \rangle_{m, \varepsilon} - \langle I_\mu | I_\nu \rangle \cdot \langle \emptyset | \Theta^\red_+ | \emptyset \rangle_{0,\varepsilon} = \langle I_\mu, I_\nu \rangle_{\Hilb^m, \varepsilon}^\red,$$
  where the right hand side is the 2-point $T^\pm$-equivariant genus-zero reduced Gromov--Witten invariant of $\Hilb([\C^2/\Z_{n+1}])$, with curve class $\beta= \varepsilon$.
\end{prop}

\begin{proof}
  We observe that both operators $\langle\ ,\ \rangle_{m, \varepsilon, +}^\sim - \langle \ | \ \rangle \cdot \langle \emptyset | \Theta^\sim_+ | \emptyset \rangle_{0,\varepsilon}$ and $\langle \ ,\ \rangle_{\Hilb^m,\varepsilon}$ satisfy the factorization result as in Lemma \ref{factorization}, and coincide on the vacuum vector in the case $m=0$. The lemmas above reduce the case $\mu=\nu$ to the case $\mu\neq \nu$, which has been proved in Proposition \ref{rubber_QH}.
\end{proof}

We also need the following compactness result, which requires a  compactly supported basis in some sense. We choose the stable basis.

\begin{lem} \label{Q_value}
  Let $\mu, \nu$ be multi-regular $(n+1)$-colored partitions. Then
  $$\left\langle \Stab_{-\fC}(I_\mu) \middle| \Theta^\red_+ \middle| \Stab_\fC (I_\nu) \right\rangle \in \Q.$$
\end{lem}

\begin{proof}
  First let's pass to the compactified and rigidified theory. Consider
  $$\left\langle \Stab_{-\fC}(I_\mu) \middle| \sigma_0(\iota_* e_i) \middle| \Stab_\fC (I_\nu) \right\rangle_{\bar\cX}^\red,$$
  where $\bar\cX$ is $\bar S \times \Pj^1$ as before.

  The invariants are defined by the degree of
  $$\ev_0^*\Stab_{-\fC}(I_\mu) \cdot \ev_\infty^*\Stab_\fC (I_\nu) \cdot \sigma_0(\iota_*e_i)\cdot [\bar\cM]^\red,$$
  which is equivalent to the degree of
  $$\left[ \Stab_{-\fC}(I_\mu)\times \Stab_\fC (I_\nu) \right] \cdot \ev_*\left( \sigma_0(\iota_*e_i) \cdot [\bar\cM]^\red \right),$$
  where $\ev:= \ev_0\times \ev_\infty$.

  View $\ev_*\left( \sigma_0(\iota_*e_i) \cdot [\bar\cM]^\red \right)$ as a correspondence, then the invariant is given by the $(I_\mu,I_\nu)$ matrix element of the correspondence
  $$\Stab^\tau_{-\fC} \circ \ev_*\left( \sigma_0(\iota_*e_i) \cdot [\bar\cM]^\red \right) \circ \Stab_\fC.$$
  Now the key observation is that $\ev_*\left( \sigma_0(\iota_*e_i) \cdot [\bar\cM]^\red \right)$ is actually a Steinberg correspondence, since it is of virtual dimension $2m$ and the objects are always supported on $\Pj^1[k]\times Z$ for some 0-dimensional $Z\subset \Hilb^m([\C^2/\Z_{n+1}])$. By Lemma \ref{properness}, the only push-forward step involved in the definition of the convolution is proper. We conclude that our invariants live in $\Q[s_1,s_2]$. By a further dimension counting, we have
  $$\left\langle \Stab_{-\fC}(I_\mu) \middle| \sigma_0(\iota_* e_i) \middle| \Stab_\fC (I_\nu) \right\rangle_{\bar\cX}^\red \in \Q.$$
  Finally we apply the factorization rule and obtain
  $$\left\langle \Stab_{-\fC}(I_\mu) \middle| \sigma_0(\iota_* e_i) \middle| \Stab_\fC (I_\nu) \right\rangle_{\cX}^\red = \left\langle \Stab_{-\fC}(I_\mu) \middle| \sigma_0(\iota_* e_i) \middle| \Stab_\fC (I_\nu) \right\rangle_{\bar\cX}^\red \in \Q,$$
  which implies the lemma by rigidification.
\end{proof}

\subsection{Punctual invariants}

\begin{prop} \label{7.10}
  For $\varepsilon= k\delta$, $k\geq 1$ and any multi-regular $(n+1)$-colored partitions $\mu, \nu$ (possibly the same),
  $$\langle I_\mu |\Theta_\mathbf{0}^\red | I_\nu \rangle_{m, \varepsilon} - \langle I_\mu | I_\nu \rangle \cdot \langle \emptyset | \Theta_\mathbf{0}^\red | \emptyset \rangle_{0,\varepsilon} = \langle I_\mu, I_\nu \rangle_{\Hilb^m, \varepsilon}^\red,$$
  where the right hand side is the 2-point $T^\pm$-equivariant genus-zero reduced Gromov--Witten invariant of $\Hilb([\C^2/\Z_{n+1}])$, with curve class $\beta= \varepsilon$.
\end{prop}

\begin{proof}
  For $\mu\neq \nu$ or $m=0$ the conclusion follows from Proposition \ref{rubber_QH}. For $n=0$ it follows from Proposition 22 of \cite{OP-DT} in the case of $\C^2\times \Pj^1$. It remains to consider the case $\mu=\nu$, $m, n\geq 1$. We follow a similar argument to the non-punctual case, and use induction on $m$.

  Let $\eta$ be a multi-regular $(n+1)$-colored partition. There exists another colored partition $\eta'$, of size $(m-1)(n+1)$, such that $\eta'$ is obtained by removing a length-$(n+1)$ border strip from $\eta$. Consider the invariant
  $$\langle I_\eta | \Theta^\red_\mathbf{0} | \fp_{-1}(1) I_{\eta'} \rangle.$$
  Write $I_\eta, I_{\eta'}$ in terms of Nakajima basis:
  $$I_\eta = \sum C^{\mu, \lambda^i}_\eta \mu[1]\prod \lambda^i(\omega_i), \qquad I_{\eta'} = \sum C^{\nu, \rho^i}_{\eta'} \nu[1] \prod \rho^i (\omega^i).$$
  Note that $\mu, \nu$ here are ordinary partitions instead of colored ones. By Lemma \ref{factorization} we have
  \begin{eqnarray*}
    && \langle I_\eta | \Theta^\red_\mathbf{0} | \fp_{-1}(1) I_{\eta'} \rangle - \langle I_\eta | \fp_{-1}(1) I_{\eta'} \rangle \langle \emptyset | \Theta^\red_\mathbf{0} | \emptyset \rangle \\
    &=& \sum C^{\mu, \lambda^i}_\eta C^{\nu, \rho^i}_{\eta'} \left\langle \mu[1]\prod \lambda^i(\omega_i) \middle| \Theta_\bzero^\red \middle| ((1)\cup\nu)[1] \prod \rho^i (\omega^i) \right\rangle \\
    &=& \sum C^{\mu, \lambda^i}_\eta C^{\nu, \rho^i}_{\eta'} \langle \mu[1] | ((1)\cup \nu)[1] \rangle \left( \left\langle \prod \lambda^i(\omega_i) \middle| \Theta_\bzero^\red \middle| \prod \rho^i (\omega^i) \right\rangle - \left\langle \prod \lambda^i(\omega_i) \middle| \prod \rho^i (\omega^i) \right\rangle \langle \emptyset | \Theta_\bzero^\red | \emptyset \rangle \right) \\
    && + \sum C^{\mu, \lambda^i}_\eta C^{\nu, \rho^i}_{\eta'} \left\langle \prod \lambda^i(\omega_i) \middle| \prod \rho^i (\omega^i) \right\rangle \left( \langle \mu[1] | \Theta^\red_\mathbf{0} | ((1)\cup \nu)[1] \rangle - \langle \mu[1] | ((1)\cup \nu)[1] \rangle \langle \emptyset | \Theta_\bzero^\red | \emptyset \rangle \right)
  \end{eqnarray*}
  By induction hypothesis and the next lemma we have
  \begin{eqnarray*}
    && \langle I_\eta | \Theta^\red_\mathbf{0} | \fp_{-1}(1) I_{\eta'} \rangle - \langle I_\eta | \fp_{-1}(1) I_{\eta'} \rangle \langle \emptyset | \Theta^\red_\mathbf{0} | \emptyset \rangle \\
    &=& \sum C^{\mu, \lambda^i}_\eta C^{\nu, \rho^i}_{\eta'} \langle \mu[1] | ((1)\cup \nu)[1] \rangle \left\langle \prod \lambda^i(\omega_i) , \prod \rho^i (\omega^i) \right\rangle^\red_{\Hilb, \mathbf{0}}  \\
    && + \sum C^{\mu, \lambda^i}_\eta C^{\nu, \rho^i}_{\eta'} \left\langle \prod \lambda^i(\omega_i) \middle| \prod \rho^i (\omega^i) \right\rangle \langle \mu[1], ((1)\cup \nu)[1] \rangle^\red_{\Hilb, \mathbf{0}}.
  \end{eqnarray*}
  Here the boldface subscript $``\bzero"$ means the ``punctual" part in the invariants, i.e. the terms with $q^\varepsilon$ with $\varepsilon_0 = \cdots = \varepsilon_n$.

  Now one can check the operator $\langle \ , \ \rangle^\red_{\Hilb, \mathbf{0}}$ also satisfies an identity as in Lemma \ref{factorization}, by computing the descendent invariants on $\Hilb(\bar S)$. We have
  \begin{eqnarray*}
    && \langle I_\eta | \Theta^\red_\mathbf{0} | \fp_{-1}(1) I_{\eta'} \rangle - \langle I_\eta | \fp_{-1}(1) I_{\eta'} \rangle \langle \emptyset | \Theta^\red_\mathbf{0} | \emptyset \rangle \\
    &=& \sum C^{\mu, \lambda^i}_\eta C^{\nu, \rho^i}_{\eta'} \left\langle \mu[1]\prod \lambda^i(\omega_i) , ((1)\cup \nu)[1] \prod \rho^i (\omega^i) \right\rangle^\red_{\Hilb, \mathbf{0}}  \\
    &=& \langle I_\eta , \fp_{-1}(1) I_{\eta'} \rangle^\red_{\Hilb, \mathbf{0}}.
  \end{eqnarray*}
  Now as in Proposition \ref{munu-induction} one observes that $\fp_{-1}(1) I_{\eta'}$ is a $\Q(s)$-linear combination of various $I_{\eta''}$ with exactly one term containing $I_\eta$. The proposition is proved.
\end{proof}

\begin{lem}
  Let $\mu, \nu\neq \emptyset$ be ordinary partitions. Then
  $$\langle \mu[1] | \Theta^\red_\mathbf{0} | \nu[1] \rangle - \langle \mu[1] | \nu[1] \rangle \langle \emptyset | \Theta_\bzero^\red | \emptyset \rangle = \langle \mu[1], \nu[1] \rangle^\red_{\Hilb, \mathbf{0}}.$$
\end{lem}

\begin{proof}
  Consider the compactification $\bar\cX = \bar S\times \Pj^1$. Since $\mu, \nu\neq \emptyset$, the invariants
  $$\langle \mu[1] | \sigma_1(\iota_* F) | \nu[1] \rangle^{\bar\cX}$$
  vanish by compactness and dimension counting. Thus by Lemma \ref{factorization} we have
  \begin{eqnarray*}
    0 &=& \langle \mu[1] | \sigma_1(\iota_* F) | \nu[1] \rangle^\cX + \sum_{j=1}^3 \langle \mu[1] | \nu[1] \rangle^\cX \langle \emptyset | \sigma_1(\iota_* F) | \emptyset \rangle^{\C^2\times \Pj^1(j)} \\
    && + \sum_{\mu^{(j)}, \nu^{(j)}} C_{\mu^{(j)}, \nu^{(j)}} \left\langle \mu^{(0)}[1] \middle| \sigma_1(\iota_* F) \middle| \nu^{(0)}[1]  \right\rangle^\cX \cdot \prod_{j=1}^3 \left\langle \mu^{(j)}[1] \middle| \nu^{(j)}[1] \right\rangle^{\C^2\times \Pj^1(j)} \\
    && + \sum_{\mu^{(j)}, \nu^{(j)}} C_{\mu^{(j)}, \nu^{(j)}} \left\langle \mu^{(0)}[1] \middle| \nu^{(0)}[1] \right\rangle^\cX \cdot \sum_{j=1}^3 \left\langle \mu^{(j)}[1] \middle| \sigma_1(\iota_* F) \middle| \nu^{(j)}[1] \right\rangle^{\C^2\times \Pj^1(j)} \\
    && \cdot \prod_{j'\neq j} \left\langle \mu^{(j')}[1] \middle| \nu^{(j')}[1] \right\rangle^{\C^2\times \Pj^1(j')},
  \end{eqnarray*}
  Let's compare this with its counterpart in the case of $\cA_n\times \Pj^1$. Consider the compactification $\bar\cX'$ of $\cA_n\times \Pj^1$. The same argument shows that $\langle \mu[1] |\sigma_1(\iota_* F)| \nu[1] \rangle^{\cA_n\times \Pj^1}$ also satisfies the equation above. Induction on $\min\{l(\mu), l(\nu)\}$ reduces the problem to a comparison of the vacuum expectations. But we have
  $$\langle \emptyset | \sigma_1(\iota_* F) | \emptyset \rangle^\cX_{0, \mathbf{0}} = \langle \emptyset | \sigma_1(\iota_* F) | \emptyset \rangle^{\cA_n\times \Pj^1}_{0, \mathbf{0}},$$
  by Corollary \ref{vacuum-expectation}. Hence we conclude
  $$\langle \mu[1] | \Theta^\red_\mathbf{0} | \nu[1] \rangle^\cX = \langle \mu[1] | \Theta^\red_\mathbf{0} | \nu[1] \rangle^{\cA_n\times \Pj^1}.$$
  The lemma follows from the correspondence \cite{MOb09-DT} between DT theory of $\cA_n\times \Pj^1$ and $\Hilb(\cA_n)$, and explicit formulas of quantum multiplications on $\Hilb(\cA_n)$ and $\Hilb([\C^2/\Z_{n+1}])$, which will be given in later sections.
\end{proof}

\subsection{3-point DT invariants}

We need some more results on the Hilbert scheme of points on $[\C^2/\Z_{n+1}]$. Let $\bJ$ be the universal ideal sheaf on $\Hilb^m([\C^2/\Z_{n+1}])\times [\C^2/\Z_{n+1}]$, and $p$ be the projection to the first factor. We also have the inertia stack $\Hilb^m([\C^2/\Z_{n+1}])\times I[\C^2/\Z_{n+1}]$ and the orbifold Chern character $\ch^\orb$. Denote by $\cH_i$ its $i$-th twisted component and also by $p$ the first projection. Let
$$D'_0:= -\frac{1}{n+1} \left( \sum_{j=0}^n c_1(\cV_j) \right),\qquad D'_i:= -\frac{2-\zeta^i-\zeta^{-i}}{n+1} \left( \sum_{j=0}^n \zeta^{-ij} c_1(\cV_j) \right), \quad 1\leq i\leq n.$$
We have the following result.

\begin{lem} \label{push_Hilb} In the full $T$-equivariant cohomology,
  $$p_* ch_3 (\bJ)= D'_0- \frac{m}{2} (s_1+s_2), \qquad p_* \ch^\orb_2(\bJ|_{\cH_{-i}})= D'_i, \quad 1\leq i\leq n.$$
\end{lem}

\begin{proof}
  Let $\tilde\cZ$ be the universal substack on $\Hilb^m([\C^2/\Z_{n+1}])\times [\C^2/\Z_{n+1}]$. It's clear that $ch_i(\bJ)= -ch_i(\cO_{\tilde\cZ})$ for $i>0$. Let's apply the orbifold Riemann-Roch in the sense of \cite{To} to the sheaves $\cO_{\tilde\cZ}\otimes \rho_{-k}$, for $0\leq k\leq n$. We have
  \begin{eqnarray*}
  ch(p_*(\cO_{\tilde\cZ} \otimes \rho_{-k})) &=& p_*\left( ch(\cO_{\tilde\cZ}) \cdot Td([\C^2/\Z_{n+1}]) \right) \\
  && + \sum_{i=1}^n p_* \left( \ch(\cO_{\tilde\cZ}|_{\cH_{-i}}) \cdot \ch_0(\rho_{-k}) \cdot \frac{Td(B\Z_{n+1})}{ch(\rho(\lambda_{-1} N^\vee))} \right).
  \end{eqnarray*}
  Take the $A^1$-part of both sides. The equality becomes
  $$c_1(\cV_k) = p_*ch_3(\cO_{\tilde\cZ}) -\frac{m}{2}(s_1+s_2) + \sum_{i=1}^n \frac{\zeta^{ik}}{2-\zeta^i-\zeta^{-i}} \cdot p_*\ch_2^\orb (\cO_{\tilde\cZ} |_{\cH_{-i}}).$$
  Therefore,
  \begin{eqnarray*}
    p_*ch_3(\cO_{\tilde\cZ}) &=& \frac{1}{n+1} \left( \sum_{j=0}^n \left( c_1(\cV_j)+ \frac{m}{2} (s_1+s_2) \right) \right) \\
    &=& -D'_0 + \frac{m}{2} (s_1+s_2),
  \end{eqnarray*}
  \begin{eqnarray*}
    p_*\ch_2^\orb (\cO_{\tilde\cZ} |_{\cH_{-i}}) &=& \frac{2-\zeta^i-\zeta^{-i}}{n+1} \left( \sum_{j=0}^n \zeta^{-ij} \left( c_1(\cV_j)+ \frac{m}{2}(s_1+s_2) \right) \right) \\
    &=& -D'_i,
  \end{eqnarray*}
  which proves the lemma.
\end{proof}

Now we arrive at a point to prove our main theorem. Previous preparations in this section allow us to prove the following partial results for the $\Theta_+$ part.

\begin{thm}
  Let $\gamma\in A^*_T(\Hilb^m([\C^2/\Z_{n+1}]))$ be the fundamental class or a divisor. Then
  $$\langle A, B, \gamma \rangle'_{m,\varepsilon} = \langle A, B, \gamma \rangle_{\Hilb^m, \varepsilon}.$$
\end{thm}

\begin{proof}
  For $\gamma=1$ the theorem follows from Corollary \ref{identity} and Theorem \ref{tube}. In the following we assume that $\gamma$ is a divisor.

  Proposition \ref{rigidification} tells us that
  $$\langle A | \sigma_1(\iota_* F) | B \rangle_{m,\varepsilon} = (D'_0\cdot \varepsilon) \langle A,B \rangle^\sim_{m,\varepsilon},$$
  $$\langle A | \sigma_0(\iota_* e_i) | B \rangle_{m,\varepsilon} = (D'_i\cdot \varepsilon) \langle A,B \rangle^\sim_{m,\varepsilon}, \quad 1\leq i\leq n.$$
  Let's compute the left hand side by the degeneration formula. Note that by Theorem \ref{deg-0-rel}, \emph{reduced} 2-point DT invariants are the same as non-reduced ones.

  \begin{eqnarray} \label{desc_deg}
    \langle A | \sigma_1 (\iota_* F) | B \rangle'_{m, \varepsilon} &=& \sum_{\varepsilon, a, b} \langle A, B, C_a \rangle'_{m, \varepsilon} g^{ab} \langle C_b | \sigma_1 (\iota_* F) \rangle'_{m,0} \\
    && + \sum_{\varepsilon' + \varepsilon''=\varepsilon, \varepsilon''\neq 0, a, b} \langle A, B, C_a \rangle'_{m, \varepsilon'} g^{ab} \langle C_b | \sigma_1 (\iota_* F) \rangle'_{m,\varepsilon''}. \nonumber
  \end{eqnarray}

  For the first term in (\ref{desc_deg}), by the description of $\Hilb^{m,0}(\cX)$ in Lemma \ref{Hilb_m_0}, we have
  \begin{eqnarray*}
  \langle C_b | \sigma_1 (\iota_* F) \rangle'_{m,0} &=& \deg \left( p_*\left(\ch_3^\orb (\bI) \cdot I^*q^*\iota_* F \cdot p^*\left[\Hilb^{m,0}(\cX)\right] \right) \cdot \ev^* C_b \right)\\
  &=& \int_{\Hilb^m([\C^2/\Z_{n+1}])\times [\C^2/\Z_{n+1}]} ch_3(\cI) \cdot p^* C_b \\
  &=& \left\langle p_* ch_3(\cI) \middle| C_b \right\rangle \\
  &=& \left\langle D'_0- \frac{m}{2} (s_1+s_2) \middle| C_b \right\rangle,
  \end{eqnarray*}
  where in the last equality we use the identity in Lemma \ref{push_Hilb}. We conclude that the first term on the right hand side of (\ref{desc_deg}) is
  $$\langle A, B, D'_0 -\frac{m}{2} (s_1+s_2) \rangle'_{m,\varepsilon} = \langle A, B, D'_0\rangle'_{m,\varepsilon}.$$

  For the second term in (\ref{desc_deg}), by Corollary \ref{identity} we have
  $$\langle C_b | \sigma_1 (\iota_* F) \rangle'_{m,\varepsilon''} = \langle C_b | \sigma_1 (\iota_* F) | (1^m)[1] \rangle'_{m,\varepsilon''} = (D'_0\cdot \varepsilon'') \langle C_b, (1^m)[1] \rangle^\sim_{m, \varepsilon''},$$
  which by Lemma \ref{factorization} forces $C_b$ to be the dual of $(1^m)[1]$, $C_a$ to be $(1^m)[1]$, and $\varepsilon'=0$, $\varepsilon''=\varepsilon$.

  Therefore the second term is
  $$\langle A | B \rangle \cdot (D'_0\cdot \varepsilon) \cdot \langle m! (1^m)[\pt], (1^m)[1] \rangle^\sim_{m, \varepsilon} = \langle A|B \rangle \cdot (D'_0\cdot \varepsilon) \langle \emptyset | \Theta | \emptyset \rangle_{0,\varepsilon}.$$

  Hence by Proposition \ref{rigidification}
  \begin{eqnarray*}
    \langle A, B, D'_0 \rangle'_{m,\varepsilon} &=& \langle A | \sigma_1 (\iota_* F) | B \rangle_{m, \varepsilon} - \langle A|B \rangle \cdot (D'_0\cdot \varepsilon) \langle \emptyset | \Theta | \emptyset \rangle_{0,\varepsilon}   \\
    &=& (D'_0\cdot \varepsilon) \langle A, B\rangle^\sim_{m,\varepsilon} - \langle A|B \rangle \cdot (D'_0\cdot \varepsilon) \langle \emptyset | \Theta | \emptyset \rangle_{0,\varepsilon}.
  \end{eqnarray*}

  Now we take $A$, $B$ in the stable basis with opposite chambers $\{\Stab_{-\fC}(I_\mu)\}$, $\{ \Stab_\fC (I_\nu)\}$ respectively. By Lemma \ref{Q_value}, Proposition \ref{7.10} and the computations on vacuum expectations, we know that the right hand side, and thus $\langle A, B, D'_0 \rangle'_{m,\varepsilon}$, take values in $(s_1+s_2)\cdot \Q$. Therefore it suffices to compute their $(s_1+s_2)$-coefficient, in the specialization $s_1+s_2=0$. By Proposition \ref{rubber_QH_red} and Proposition \ref{7.10} we have
  \begin{eqnarray*}
    (s_1+s_2)^{-1}\cdot \langle A, B, D'_0 \rangle'_{m,\varepsilon} &=& \left. \left( (D'_0\cdot \varepsilon) \langle A, B\rangle^{\sim, \red}_{m,\varepsilon} - \langle A|B \rangle \cdot (D'_0\cdot \varepsilon) \langle \emptyset | \Theta^\red | \emptyset \rangle_{0,\varepsilon} \right) \right|_{s_1=s, s_2=-s} \\
    &=&  (D'_0\cdot \varepsilon) \langle A, B \rangle^\red_{\Hilb^m, \varepsilon} \\
    &=& \langle A, B, D'_0 \rangle^\red_{\Hilb^m, \varepsilon}.
  \end{eqnarray*}

  Moreover, by \cite{MOk12} we know that $\langle A, B, D'_0 \rangle_{\Hilb^m, \varepsilon}$ also lives in $(s_1+s_2)\cdot \Q$ and thus equals to $\langle A, B, D'_0 \rangle_{\Hilb^m, \varepsilon}^\red$. Hence we conclude
  $$\langle A, B, D'_0 \rangle'_{m,\varepsilon} = \langle A, B, D'_0 \rangle_{\Hilb^m, \varepsilon}.$$

  Applying the same procedure to $\sigma_0(\iota_* \theta_i)$, we have for $1\leq i\leq n$,
  $$\langle A, B, D'_i \rangle'_{m,\varepsilon} = \langle A, B, D'_i \rangle_{\Hilb^m, \varepsilon}.$$
  Since $D'_0, \cdots, D'_n$ form a basis for the divisor group, the same identity holds for $D'_i$ replaced by any divisor. The proposition is proved.
\end{proof}

\subsection{Quantum cohomology of $\Hilb^m([\C^2/\Z_{n+1}])$}

To write down an explicit formula for the 3-point functions, we need the results of Maulik--Okounkov \cite{MOk12} on the quantum cohomology of quiver varieties.

Recall that to each quiver $Q$ and the associated quiver variety $X$, there is a Lie algebra $\fg_Q$, with a Cartan subalgebra $\fh$ identified (more precisely, in general surjects onto) with $H^2(X,\Z)$ and the weight space $\fh^*$ identified with (in general embeds into) $H_2(X,\Z)$. A basis of $H^2(X,\Z)$ can be taken as $\{c_1(\cV_i)\}$ where $\cV_i$ are tautological bundles, and a basis of $H_2(X,\Z)$ can be taken as a certain choice of simple roots of $\fg_Q$. $\fg_Q$ acts on the cohomology of $X$, by certain Steinberg correspondences. In particular there are operators given by $e_\alpha$, where $\alpha$ are roots of $\fg_Q$.

\begin{thm}[Theorem 1.3.2 of \cite{MOk12}] \label{thm_QH}
Let $\lambda:= \sum_{i\in I} \lambda_i c_1(\cV_i)$. The quantum multiplication by $\lambda$ is given by
$$\lambda * (\ ) = \lambda \cup (\ ) -\hbar \sum_{\theta\cdot \alpha>0} (\lambda, \alpha) \frac{q^\alpha}{1-q^\alpha} e_\alpha e_{-\alpha} + \cdots,$$
where $\hbar$ is the equivariant weight of the holomorphic symplectic form, $\theta$ is the stability condition chosen in the definition of the quiver variety, $q$ is the quantum parameter, and the dots stand for a multiple of the identity, which can be determined by the identity property $1*\gamma=\gamma$.
\end{thm}

The quantum parameter $q$ here is actually subject to a certain sign modification $q\mapsto q_\kappa$, which we call the $\kappa$-modification, and will only specify in our special case of cyclic quiver variety. The weight space of $\widehat\fgl(n+1)$ is spanned by the simple roots $\alpha_1, \cdots, \alpha_n$ of $\fgl(n+1)$, and an imaginary root $\delta$. There is another simple root $\alpha_0$ defined as $\alpha_0= \delta - \alpha_1 - \cdots - \alpha_n$. Let $\alpha$ be a root. Recall that the stability condition for $\Hilb^m([\C^2/\Z_{n+1}])$ as a quiver variety is $\theta_i>0$, $0\leq i\leq n+1$. In our case, $q\mapsto q_\kappa$ is defined as
$$q^\delta \mapsto -q^\delta, \qquad q^{\alpha_i} \mapsto q^{\alpha_i}, \qquad i\geq 1.$$

If $\alpha= k\delta + (\alpha_i +\cdots +\alpha_{j-1})$, then
$$\theta\cdot \alpha = k(n+1) + (j-i),$$
for which $\theta\cdot \alpha>0$ if and only if $k>0$.

If $\alpha= k\delta - (\alpha_i +\cdots +\alpha_{j-1})$, then
$$\theta\cdot \alpha = k(n+1) - (j-i),$$
for which $\theta\cdot \alpha>0$ if and only if $k\geq 0$.

If $\alpha= k\delta$, then
$$\theta\cdot \alpha = k(n+1),$$
for which $\theta\cdot \alpha>0$ if and only if $k>0$.

We conclude that the roots satisfying $\theta\cdot \alpha>0$ are exactly the positive roots of $\widehat{\fgl}(n+1)$.

Let $\alpha_{ij}:=\alpha_i+\cdots +\alpha_{j-1}$. Let $M_D$ and $M_D^{cl}$ be the operators defined by quantum and classical multiplications by the divisor $D$. Recall that
$$D_0 = c_1(\cV_0), \qquad D_l = c_1(\cV_l)- c_1(\cV_0), \qquad l\geq 1.$$
We now compute by Theorem \ref{thm_QH}.

For $l\geq 1$,
\begin{eqnarray*}
  M_{D_l}-M_{D_l}^{cl}
  &=& -(s_1+s_2) \left( \sum_{k>0} \sum_{1\leq i\leq l<j\leq n+1} \frac{q_\kappa^{k\delta + \alpha_{ij}}}{1-q_\kappa^{k\delta+\alpha_{ij}}} e_{k\delta+\alpha_{ij}} e_{-k\delta-\alpha_{ij}} \right. \\
&& \left. + \sum_{k\geq 0} \sum_{1\leq i\leq l<j\leq n+1} \frac{q_\kappa^{k\delta - \alpha_{ij}}}{1-q_\kappa^{k\delta-\alpha_{ij}}} e_{k\delta-\alpha_{ij}} e_{-k\delta+\alpha_{ij}} \right) + \cdots \\
&=& -(s_1+s_2) \left( \sum_{k>0} \sum_{1\leq i\leq l<j\leq n+1} \frac{(-q^\delta)^k q^{\alpha_{ij}}}{1-(-q^\delta)^k q^{\alpha_{ij}}} :e_{k\delta+\alpha_{ij}} e_{-k\delta-\alpha_{ij}}: \right. \\
&& \left. + \sum_{k\geq 0} \sum_{1\leq i\leq l<j\leq n+1} \frac{(-q^\delta)^k q^{-\alpha_{ij}}}{1-(-q^\delta)^k q^{-\alpha_{ij}}} :e_{k\delta-\alpha_{ij}} e_{-k\delta+\alpha_{ij}}: \right),
\end{eqnarray*}
where the normally ordered product is defined as
$$:e_{k\delta+\alpha_{ij}} e_{-k\delta-\alpha_{ij}}: = \left\{ \begin{aligned}
  &e_{k\delta+\alpha_{ij}} e_{-k\delta-\alpha_{ij}}, \quad k<0 \\
  &e_{-k\delta-\alpha_{ij}} e_{k\delta+\alpha_{ij}}, \quad k\geq 0.
  \end{aligned} \right.$$

For $l=0$, before the $\kappa$-modification, we compute
\begin{eqnarray*}
  \sum_{\theta\cdot \alpha>0} (D_0, \alpha) \frac{q^\alpha}{1-q^\alpha} e_\alpha e_{-\alpha} + \cdots &=& \sum_{k>0} \sum_{1\leq i<j\leq n+1} \frac{k q^{k\delta + \alpha_{ij}}}{1-q^{k\delta+\alpha_{ij}}} :e_{k\delta+\alpha_{ij}} e_{-k\delta-\alpha_{ij}}: \\
&& + \sum_{k> 0} \sum_{1\leq i<j\leq n+1} \frac{k q^{k\delta - \alpha_{ij}}}{1-q^{k\delta-\alpha_{ij}}} :e_{k\delta-\alpha_{ij}} e_{-k\delta+\alpha_{ij}}: \\
&& + \sum_{k>0} \frac{k q^{k\delta}}{1-q^{k\delta}} e_{-k\delta} e_{k\delta} + \cdots.
\end{eqnarray*}
For convenience we switch the order of $e_{k\delta}$ and $e_{-k\delta}$, which does not affect the result since the difference would be a scalar operator and can be absorbed into the ``dots" term.

By definition the last term is a sum of products of the dual bases between $\fg_{k\delta}$ and $\fg_{-k\delta}$, with respect to the canonical bilinear form. This bilinear form is given by
$$(\alpha_i(k) \mid \alpha_j(-k)) = (\alpha_i | \alpha_j) = C_{ij},$$
where $C$ is the Cartan matrix. Thus the dual basis is
$$\alpha_i(-k)^* = \sum_j \alpha_j(k) \cdot C^{-1}_{ij} = \sum_j \fp_k(E_j) \cdot C^{-1}_{ij} = -\fp_k(\omega_i),$$
and we have
\begin{eqnarray*}
e_{-k\delta} e_{k\delta} &=& \frac{1}{n+1}\Id(-k)\Id(k) + \sum_{i=1}^n \alpha_i(-k) \alpha_i(-k)^* \\
&=& -\fp_{-k}(1)\fp_{k}(\pt) - \sum_{i=1}^n \fp_{-k}(E_i)\fp_{k}(\omega_i).
\end{eqnarray*}

To determine the ``dots" term, we apply the operator to the fundamental class $1$, which is $\fp_{-1}(1)^m 1$ on level $m$.
\begin{eqnarray*}
\sum_{\theta\cdot \alpha>0} (D_0, \alpha) \frac{q_\kappa^\alpha}{1-q_\kappa^\alpha} e_\alpha e_{-\alpha} \cdot \fp_{-1}(1)^m 1 &=& -\sum_{k\geq 1} \frac{k q_\kappa^{k\delta}}{1-q_\kappa^{k\delta}} \fp_{-k}(1)\fp_{k}(\pt) \fp_{-1}(1)^m 1 \\
&=& - \frac{q^{\delta}}{1-q^{\delta}} \fp_{-1}(1)\fp_{1}(\pt) \fp_{-1}(1)^m 1 \\
&=& -\frac{mq^{\delta}}{1-q^{\delta}} \fp_{-1}(1)^m 1,
\end{eqnarray*}
which is exactly $-\frac{q^{\delta}}{1-q^{\delta}}$ times the energy operator
$$L_0 = \sum_{k\geq 1} (\fp_{-k}(1)\fp_k(\pt) + \fp_{-k}(E_i)\fp_k(\omega_i)).$$
Thus the pure quantum multiplication $M_{D_0}-M_{D_0}^{cl}$ is $-(s_1+s_2)$ times the following
\begin{eqnarray*}
  && \sum_{k>0} \sum_{1\leq i<j\leq n+1} \frac{kq^{k\delta + \alpha_{ij}}}{1-q^{k\delta+\alpha_{ij}}} :e_{k\delta+\alpha_{ij}} e_{-k\delta-\alpha_{ij}}: + \sum_{k>0} \sum_{1\leq i<j\leq n+1} \frac{kq^{k\delta - \alpha_{ij}}}{1-q^{k\delta-\alpha_{ij}}} :e_{k\delta-\alpha_{ij}} e_{-k\delta+\alpha_{ij}}: \\
  && + \sum_{k\geq 1} \frac{kq^{k\delta}}{1-q^{k\delta}} \left( -\fp_{-k}(1)\fp_k(\pt) - \sum_{i=1}^n \fp_{-k}(E_i)\fp_k(\omega_i) \right) + \frac{q^{\delta}}{1-q^{\delta}} L_0 \\
  &\mapsto& \sum_{k\in \Z} \sum_{1\leq i<j\leq n+1} \frac{k (-q^\delta)^k  q^{\alpha_{ij}}}{1-(-q^\delta)^k q^{\alpha_{ij}}} :e_{k\delta+\alpha_{ij}} e_{-k\delta-\alpha_{ij}}: + \sum_{k>0} \sum_{1\leq i<j\leq n+1} \frac{k (-q^\delta)^k  q^{\alpha_{ij}}}{1-(-q^\delta)^k q^{\alpha_{ij}}} :e_{k\delta-\alpha_{ij}} e_{-k\delta+\alpha_{ij}}: \\
  && - \sum_{k\geq 1} \left( \frac{k (-q^\delta)^k}{1-(-q^\delta)^k} - \frac{-q^{\delta}}{1-(-q^{\delta})} \right) \left( \fp_{-k}(1)\fp_k(\pt) + \sum_{i=1}^n \fp_{-k}(E_i)\fp_k(\omega_i) \right),
\end{eqnarray*}

In \cite{MOb09-DT} explicit formulas are obtained for the relative DT invariants of $\cA_n\times \Pj^1$. Generating functions for $\cA_n\times \Pj^1$ are in parameters $Q, s_1, \cdots, s_n$. Comparison with their formulas yields the following crepant resolution result for relative DT theory. The cohomology of $A_{T^\pm}^*(\Hilb^m([\C^2/\Z_{n+1}]))$ and $A_{T^\pm}^*(\Hilb^m(\cA_n))$ are identified as in Section 2.3. Recall that the parameters $q^{\alpha_i}$, $0\leq i\leq n$ are identified with $q_i$, $0\leq i\leq n$.

\begin{thm}
  Let $\gamma\in A_{T^\pm}^*(\Hilb^m([\C^2/\Z_{n+1}]))$ be the fundamental class or a divisor. Then
  $$\langle A, B, \gamma \rangle'_{\cX, m, \varepsilon} = \langle A, B, \gamma \rangle'_{\cA_n\times \Pj^1, \chi, (m, \beta)},$$
where
$$\chi = m + \varepsilon_0, \qquad \beta = \sum_{j=1}^n (\varepsilon_j-\varepsilon_0) E_j \in H_2(\cA_n, \Z). $$
\end{thm}

In particular, if we define the relative DT generating function for $\cA_n\times \Pj^1$ as
$$\langle \xi_1, \cdots, \xi_r \rangle_{\cA_n\times \Pj^1, m} := \sum_{\chi, \beta} Q^\chi s_1^{(\omega_1, \beta)} \cdots s_n^{(\omega_n, \beta)} \langle \xi_1, \cdots, \xi_r \rangle_{\cA_n\times \Pj^1, \chi, (m, \beta)},$$
we have the following.

\begin{cor}
Let $\gamma\in A_{T^\pm}^*(\Hilb^m([\C^2/\Z_{n+1}]))$ be the fundamental class or a divisor. Then
  $$\langle A, B, \gamma \rangle'_{\cX, m} = Q^{-m} \langle A, B, \gamma \rangle'_{\cA_n\times \Pj^1, m},$$
under the change of variables
$$Q\mapsto q_0 q_1 \cdots q_n, \qquad s_i \mapsto q_i, \qquad i\geq 1$$
and analytic continuation.
\end{cor}

3-point DT invariants also define a product on the cohomology of Hilbert scheme of points as
$$\langle A, B*_{DT} C\rangle:= \langle A, B, C \rangle'_m.$$
It is also a $q$-deformation of the usual cup product. We have the following corollary.

\begin{cor}
  Let $\gamma$ be $1$ or a divisor on $\Hilb^m([\C^2/\Z_{n+1}])$. The operator $\gamma*_{DT}$ is the same as $M_\gamma$.
\end{cor}

Under the nondegeneracy conjecture as in \cite{MOb09-DT}, we can equate the full DT theory and quantum cohomology theory.

\begin{conj}
  The joint eigenspaces for the operators $M_{D_l}$, $0\leq l\leq n$ are 1-dimensional for all $m>0$.
\end{conj}

\begin{cor*}
  Under the conjecture, the DT ring defined as above is generated by the divisors $D_l$, $0\leq l\leq n$. Moreover, the DT ring structure is identical to that of the small quantum cohomology of $\Hilb^m([\C^2/\Z_{n+1}])$.
\end{cor*}

\section{Future works and relations to other theories}

There is a beautiful diagram of theories as mentioned in the introduction.
$$\xymatrix{
& \text{QH}(\Hilb(\cA_n)) \ar@{-}[dl] \ar@{-}[dr] \ar@{-}'[d][dd] & \\
\text{GW}(\cA_n\times \Pj^1) \ar@{-}[rr] \ar@{--}[dd] & & \text{DT}(\cA_n\times \Pj^1) \ar@{-}[dd] \\
& \text{QH}(\Hilb([\C^2/\Z_{n+1}])) \ar@{--}[dl] \ar@{-}[dr] & \\
\text{GW}([\C^2/\Z_{n+1}]\times \Pj^1) \ar@{--}[rr] & & \text{DT}([\C^2/\Z_{n+1}]\times \Pj^1),}$$
There are several directions we would like to explore further starting from our work on the DT theory of $[\C^2/\Z_{n+1}]\times \Pj^1$.

\begin{enumerate}[$\bullet$]
\setlength{\parskip}{1ex}

\item We expect that the relative Gromov--Witten theory of $[\C^2/\Z_{n+1}]\times \Pj^1$ also corresponds to the quantum cohomology of $\Hilb([\C^2/\Z_{n+1}])$, and thus would complete the above diagram. The result would provide a new example of orbifold relative GW/DT correspondence.

\item Raised by Ruan \cite{Ruan, BG}, the crepant resolution conjecture states that the orbifold GW theory of a Gorenstein orbifold should be equivalent to the GW theory of its crepant resolution (if exists), possibly up to change of variables and analytic continuation. The crepant resolution conjecture for DT theory is raised by Bryan--Graber \cite{Br}, with the Hard-Lefschetz condition.

    Our result, compared with the work of Maulik--Oblomkov \cite{MOb09-DT}, can be viewed as a crepant resolution correspondence for relative DT theory. It turns out that the correspondence is just given by wall crossing to a different chamber in the root space associated to the cyclic quiver variety. The invariants, or operators, are related to each other without change of variables but with analytic continuation.

\item Let $\mathcal{C}$ be a 1-dimensional DM stack, and $L_1$, $L_2$ be line bundles on $\mathcal{C}$. The total space of $L_1\oplus L_2$ can be viewed as an orbifold local curve. For those orbifold local curve with transversal $A_n$-singularities, i.e. locally looks like $[\C^2/\Z_{n+1}]$, as an analogy to \cite{BP, OP-DT}, we expect that there exists a TQFT formalism which determines the GW and DT theory of $L_1\oplus L_2$. This will be pursued in future works.
\end{enumerate}

\bibliographystyle{plain}
\bibliography{reference}

\end{document}